\documentclass[12pt]{amsart}

\usepackage{tcolorbox}
\usepackage{xcolor}

\usepackage{pdfsync}
\usepackage{geometry}
\usepackage{lscape}
\usepackage{graphicx}
\usepackage{epstopdf}    
\usepackage[all]{xy}
\usepackage{amsmath,amsthm,amssymb,enumerate}
\usepackage{latexsym}
\usepackage{amscd}
\usepackage{amssymb}

\numberwithin{equation}{section}

\theoremstyle{plain}  
\newtheorem{thm}[equation]{Theorem}
\newtheorem{prop}[equation]{Proposition}

\newtheorem{lemma}[equation]{Lemma}

\theoremstyle{definition}  
\newtheorem{defn}[equation]{Definition}

\newtheorem{remark}[equation]{Remark}

\newtheorem{fact}[equation]{Fact}

\newtheorem{summary}[equation]{Summary}
\newtheorem{convention}[equation]{Conventions}
\newtheorem{relation}[equation]{Relation}

\newcommand{\ep}{\epsilon}

\newcommand{\ra}{\rightarrow}

\newcommand{\lra}{\longrightarrow}

\newcommand{\PP}{\mathbb P}

\newcommand{\pntr}{\hat{p}}
\newcommand{\Pntr}{\hat{P}}

\newcommand{\Z}{\mathbb Z}
\newcommand{\C}{\mathbb C}

\newcommand{\cp}{\C \PP^\infty}
\newcommand{\ee}{ER(2)}
\newcommand{\ww}{w_{I,\ep}}

\newcommand{\Zp}{\Z_{(2)}}
\newcommand{\Zq}{\Z/(2)}

\newcommand{\chat}{\hat{c}}
\newcommand{\uhat}{\hat{u}}
\newcommand{\vhat}{\hat{v}}

 \newcommand{\voe}{v_2^{o/e}}

 \newcommand{\Smash} {\wedge}

 \newcommand{\iso} {\cong}

\newcommand{\A}{{\bf property A}}

\thanks{
The first author was supported in part by the NSF through grant DMS 1307875.
}

\thanks{
The third author thanks Dimitri Ara and the 
Institute of Mathematics of Marseille at the
University of Aix-Marseille for their hospitality 
during the writing of this paper.
}


\begin{document}

\title{The $ER(2)$-cohomology of ${\sf X}^n \cp$ and $BU(n)$}

\author{Nitu Kitchloo}
\address{Department of Mathematics, Johns Hopkins University, 
Baltimore, Maryland, USA}
\email{nitu@math.jhu.edu}
\author{Vitaly Lorman}
\address{Department of Mathematics, University of Rochester, 
Rochester, New York, USA}
\email{vlorman@ur.rochester.edu}
\author{W. Stephen Wilson}
\address{Department of Mathematics, Johns Hopkins University, 
Baltimore, Maryland, USA}
\email{wsw@math.jhu.edu}

\begin{abstract}
We continue the development of the computability of the second real
Johnson-Wilson theory.
As
$\ee$ is not complex orientable, this gives some difficulty even
with basic spaces.  
In this paper we compute the second real Johnson-Wilson theory
for products of infinite complex projective spaces and for the
classifying spaces for the unitary groups.
\end{abstract}

\date{\today}

\maketitle

\section{Introduction}

The $p=2$ Johnson-Wilson theory, \cite[Remark 5.13]{JW2}, $E(n)$,
has coefficients
\[
E(n)^* \iso \Zp[v_1,v_2,\ldots,v_n^{\pm 1}]
\]
with the degree of $v_k$ equal to $-2(2^k-1)$.
There is a $\Zq$ action on $E(n)$ coming from complex conjugation.
The real Johnson-Wilson theory, $ER(n)$, is the homotopy fixed
points of $E(n)$.  This was initially studied by Hu and Kriz
in \cite{HK}.  Since then the theories have been studied intensively
and applied to the problem of non-immersions of real projective space.
(\cite{NituER2,NituP,NituP2,Nitufib,Kitch-Wil-split,Kitch-Wil-BO,
Kitch-Lor-Wil-ERn,Kitch-Lor-Wil-CPn,Kitch-Lor-Wil-Mult,Vitaly,Romie})

The first theory, $ER(1)$,
is just $KO_{(2)}$, and
it was a decades long process of computing the details of
the $KO$-(co)homology of $\cp$, finally ending in \cite{Yama3}.
The second theory, $ER(2)$, is, by \cite{HM},
closely related to $TMF_0(3)$ (the same after a suitable completion).  
The second author computed all $ER(n)^*(\cp)$ in complete detail, \cite{Vitaly}.
This is
already
much more than has been done with $TMF_0(3)$.

The fibre of the inclusion, $ER(n) \lra E(n)$ is $\Sigma^{2(2^n-1)^2-1}ER(n)$
from \cite{Nitufib}.  This gives a Bockstein spectral sequence
from $E(n)^*(X)$ to $ER(n)^*(X)$.
In this paper we are concerned with $ER(2)$, so we have the map $x:\Sigma^{17}ER(2) \ra
ER(2)$.  This map has $2x = 0 = x^7$.  
The resulting Bockstein spectral sequence just measures $x^i$-torsion.
We use the untruncated version. 
That just means that $d_1$ detects all of the $x^1$-torsion generators and $E_2$ is
what is left after you throw them all away.  In our cases, we only have $d_1$,
$d_3$, and $d_7$, so $E_2 = E_3$.  When we compute $d_3$, it gives us all the
$x^3$-torsion, but then we throw it all away to get our $E_4 = E_5 = E_6  = E_7$.
Our $d_7$ gives the $x^7$-torsion and leaves us with $E_8 = 0$.

Our goal here is to give a computation of this
Bockstein spectral sequence 
for $X = 
{\sf X}^n \cp$ and $ BU(n)$, computing $ER(2)^*(-)$ from $E(2)^*(-)$.
The computation is accomplished by going
through an   auxiliary spectral sequence to compute
$d_1$.  Once that is done, $d_3$ and $d_7$ follow.

Our actual computations are carried out with $\Smash^n \cp$ and
$MU(n)$ because the product and $BU(n)$ can be recovered
from the stable splittings, (e.g. $BU(n) = MU(n)\vee BU(n-1)$, 
\cite{MitPrid}).

There is a special element,
$\vhat_2 \in ER(2)^{48}$ that maps to $v_2^{-8} \in E(2)^{48}$.  It is the
periodicity element for $ER(2)$ and it makes our bookkeeping
easier if we do away with it once and for all now by setting $\vhat_2 = 1$,
and, in $E(2)^*$, the corresponding $v_2^{-8} = 1$.
This makes our theories graded over $\Z/(48)$.

There are also elements
$\vhat_1 \in ER(2)^{16}$
that maps to $v_1 v_2^{-3} \in E(2)^{16}$
and
$w \in ER(2)^{-8}$ mapping to $\vhat_1 v_2^4 = v_1 v_2 \in E(2)^{-8}$.

The theory $E(2)^*(-)$ is a complex orientable theory so 
$E(2)^*(\C \PP^\infty) = E(2)^*[[u]]$ where $u$ is of degree 2.
The only adjustment needed here is to define $\uhat = u v_2^3$,
of degree -16.
We write $E(2)^*(\C \PP^\infty) =
E(2)^*[[\uhat]]$.  Since $v_2$ is a unit, this is not a problem.

We also need the complex conjugate of $\uhat$, $c(\uhat)$.  
There is a 
class,
$\pntr \in ER(2)^{-32}(\C \PP^\infty)$,
that maps to
$\uhat\, c(\uhat) \in E(2)^{-32}(\C \PP^\infty)$.
This is a 
modified 
first Pontryagin class.

We can generalize this to $BU(n)$.  Because $E(2)$ is a complex
oriented theory, we have
$$
E(2)^*(BU(n)) \iso E(2)^*[[c_1,\ldots,c_n]].
$$
Again, we need to modify the generalized Conner-Floyd Chern
classes to $\chat_k = v_2^{3k} c_k$, putting them in degree
$-16k$.
We also have elements
$\Pntr_k \in ER(2)^{-32k}(BU(n))$ that
map to $\chat_k \, c(\chat_k)$ in $E(2)^{-32k}(BU(n))$.
These are modified Pontryagin classes, and they are
permanent cycles in our spectral sequence.

Although we compute all of $ER(2)^*(-)$ for $\Smash^n \cp$ and $MU(n)$,
it was not deemed sufficiently seductive to put our description of
the $x^1$-torsion generators in the introduction.  

We let $\uhat_i$ be our $\uhat$ associated with the $i$-th term in
the smash product of the $\cp$.  Similarly, with $\pntr_i$.
The clean results we can state nicely are presented in the next theorems.
Keep in mind that because we use an auxiliary spectral sequence to
compute $d_1$, our results are stated in terms of associated graded
versions of $E_i$.

\begin{thm}
	\label{smash}
	For the Bockstein spectral sequence going from $E(2)^*(X)$ to $ER(2)^*(X)$,
	we have associated graded versions of $E_i$ as follows:
$E_1 = $
$$
E(2)^*(\Smash^n \cp) \iso E(2)^*
[[\uhat_1,\uhat_2,\ldots,\uhat_n]]
\{\uhat_1 \uhat_2 \cdots \uhat_n \}
$$
$$
= 
\Zp[\vhat_1]
[[\uhat_1,\uhat_2,\ldots,\uhat_n]]
\{ v_2^{0-7}
\uhat_1 \uhat_2 \cdots \uhat_n \}
$$
$E_2 = E_3 = $
$$
\Zq[\pntr_n]\{v_2^{0,2,4,6} \pntr_1 \pntr_2 \ldots \pntr_n \}
$$
The $x^3$-torsion generators are represented by 
$$
\Zq[\pntr_n]\{v_2^{0,4} \pntr_1 \pntr_2 \ldots \pntr_n^2 \}
$$
$E_4 = E_5 = E_6 = E_7 = $
$$
\Zq\{v_2^{0,4} \pntr_1 \pntr_2 \ldots \pntr_n \}
$$
The $x^7$-torsion generator is represented by 
$$
\Zq\{ \pntr_1 \pntr_2 \ldots \pntr_n \}
$$
\end{thm}
\begin{thm}
	\label{mu}
	For the Bockstein spectral sequence going from $E(2)^*(X)$ to $ER(2)^*(X)$,
	we have associated graded versions of $E_i$ as follows:
$E_1 = $
$$
E(2)^*(MU(2n)) \iso E(2)^*
[[\chat_1,\chat_2,\ldots,\chat_{2n}]]
\{\chat_{2n} \}
$$
$$
= 
\Zp[\vhat_1]
[[\chat_1,\chat_2,\ldots,\chat_{2n}]]
\{ v_2^{0-7}
\chat_{2n} \}
$$
$E_2 = E_3 = $
$$
\Zq[\vhat_1][[\Pntr_2,\Pntr_4,\ldots,\Pntr_{2n}]]\{v_2^{0,2,4,6} \Pntr_{2n} \}
$$
The $x^3$-torsion generators are represented by 
$$
\Zq[\vhat_1][[\Pntr_2,\Pntr_4,\ldots,\Pntr_{2n}]]\{\vhat_1 v_2^{0,4} \Pntr_{2n} \} =
$$
$$
\Zq[\vhat_1][[\Pntr_2,\Pntr_4,\ldots,\Pntr_{2n}]]\{\vhat_1  \Pntr_{2n}, w \Pntr_{2n} \}
$$
$E_4 = E_5 = E_6 = E_7 = $
$$
\Zq[\Pntr_2,\Pntr_4,\ldots,\Pntr_{2n}]\{ v_2^{0,4} \Pntr_{2n} \}
$$
The $x^7$-torsion generators are represented by 
$$
\Zq[\Pntr_2,\Pntr_4,\ldots,\Pntr_{2n}]\{  \Pntr_{2n} \}
$$
\end{thm}
\begin{thm}
	\label{muodd}
	For the Bockstein spectral sequence going from $E(2)^*(X)$ to $ER(2)^*(X)$,
	we have associated graded versions of $E_i$ as follows:
$E_1 = $
$$
E(2)^*(MU(2n+1)) \iso E(2)^*
[[\chat_1,\chat_2,\ldots,\chat_{2n+1}]]
\{\chat_{2n+1} \}
$$
$$
= 
\Zp[\vhat_1]
[[\chat_1,\chat_2,\ldots,\chat_{2n+1}]]
\{ v_2^{0-7}
\chat_{2n+1} \}
$$
$E_2 = E_3$, for $0 \le b < n$
$$
\Zq[\Pntr_2,\Pntr_4,\ldots,\Pntr_{2b},\Pntr_{2b+1},\Pntr_{2b+3},\ldots,\Pntr_{2n+1}]
\{v_2^{0,2,4,6} \Pntr_{2b+1} \Pntr_{2n+1} \}
$$
and
$$
\Zq[\Pntr_2,\Pntr_4,\ldots,\Pntr_{2n},\Pntr_{2n+1}]
\{v_2^{0,2,4,6}  \Pntr_{2n+1} \}
$$
The $x^3$-torsion generators are represented by 
$$
\Zq[\Pntr_2,\Pntr_4,\ldots,\Pntr_{2b},\Pntr_{2b+1},\Pntr_{2b+3},\ldots,\Pntr_{2n+1}]
\{v_2^{0,4} \Pntr_{2b+1} \Pntr_{2n+1} \}
\quad 0 \le b \le n
$$
$E_4 = E_5 = E_6 = E_7 = $
$$
\Zq[\Pntr_2,\Pntr_4,\ldots,\Pntr_{2n}]
\{v_2^{0,4}  \Pntr_{2n+1} \}
$$
The $x^7$-torsion generators are represented by 
$$
\Zq[\Pntr_2,\Pntr_4,\ldots,\Pntr_{2n}]
\{  \Pntr_{2n+1} \}
$$
\end{thm}

The elements $\vhat_1$, $w$, $\pntr_i$, and $\Pntr_i$ all exist in the appropriate
$ER(2)^*(X)$.
It is worth noting that 
all of the $x^3$-torsion generators are well-defined in
$ER(2)^*(MU(2n))$ (likewise with the $x^7$-torsion generators in all three cases).
Consequently, new elements don't have to
be created and named.  
We often deal only with elements in degrees $16*$.  To see these, just modify
the statements in the theorems to eliminate the $v_2^{2,4,6}$.
In fact, we can handle elements in degrees $8*$ quite easily.  In the case
of the above theorems, just keep the $v_2^{0,4}$ and eliminate the $v_2^{2,6}$.
By definition, the $x^i$-torsion generators inject to $E(2)^*(X)$.

The following is useful for computations and relations.

\begin{thm}
	\label{inj}
	For $X = \Smash^n \cp$ and $MU(n)$, $ER(2)^{8*}(X) \ra E(2)^{8*}(X)$ injects.
\end{thm}

\begin{remark}
	In the kernel of
	$ER(2)^{4*}(\Smash^n \cp) \ra E(2)^{4*}(\Smash^n \cp)$, 
	there is only one element, 
	namely,
	$x^4 \pntr_1 \pntr_2 \ldots \pntr_n$.
	Similarly, in degrees $(8*-6)$ we have only
	$x^6 \pntr_1 \pntr_2 \ldots \pntr_n$.
\end{remark}

We do our general preliminaries in section \ref{preliminaries}.
In section \ref{sketch} we sketch out our approach in both
cases in rather general terms to give some idea of how we go
about our computations.
We define a crucial filtration in section \ref{secfilt}.
Then we spend a few sections doing the computation for $\Smash^n \cp$.
When that is done, we begin preliminaries for $BU(n)$ in
section \ref{startbu}.
We do the main calculation for $MU(n)$ starting in section
\ref{dobu} going to the end of the paper.

\section{Preliminaries}
\label{preliminaries}

There are many ways to describe $ER(2)^*$, but we will stick
mainly with the description given in
\cite[Remark 3.4]{Kitch-Wil-BO}.

We have traditionally given the name $\alpha$ to the element 
$\vhat_1$, but this is gradually being phased out. 
We also have elements $\alpha_i$, $0 < i < 4$,
with degree $-12i$.  We often extend this notation to $\alpha_0 = 2$.  
These elements map to $2v_2^{2i} \in E(2)^*$.   
For the last  non-torsion algebra generator,
we have $w$ of degree -8, which
maps to $\vhat_1 v_2^4 = v_1 v_2 \in E(2)^*$.

Torsion is generated by the element $x \in ER(2)^{-17}$.  It has $2x=0$
and $x^7 = 0$.  
Keep in mind that $ER(2)^*$ is 48 periodic.
We use, for efficient notation, $x^{3-6} = \{x^3, x^4, x^5, x^6 \}$
and $R\{a,b\}$ for the free $R$-module on $a$ and $b$.

\begin{fact}\cite[Proposition 2.1]{NituER2}
\label{ffact}
$ER(2)^*$ is:
\[
 \Zp[\vhat_1] \{ 1, w, \alpha_1, \alpha_2, \alpha_3 \} 
\quad
\text{with}
\quad
2w = \alpha \alpha_2 = \vhat_1 \alpha_2
\]
\[
\Zq[\vhat_1] \{ x^{1-2},  x^{1-2} w \}
\qquad
\Zq \{ x^{3-6} \}.
\]
\end{fact}

\begin{remark}
There is some arbitrariness about this description in terms
of $\vhat_1 = \alpha$.  Because $\alpha^2 = w^2$, we could just as
easily have used $w$ to describe $ER(2)^*$.
\end{remark}

\begin{remark}
A major theme in this paper will be to look at elements in degrees $16*$ (and sometimes even $8*$).
We have 
$ER(2)^{16*} = \Zp [ \vhat_1]$.
In addition, the $x^1$-torsion generators in degree $16*$ are given by
$\Zp [ \vhat_1]\{2\}$, the $x^3$-torsion generators, $\Zq[\vhat_1]\{\vhat_1\}$, and
the only $x^7$-torsion generator is $\Zq$.
\end{remark}

The fibration $\Sigma^{17}ER(2) \lra  ER(2) \lra E(2)$
gives rise to an 
exact couple and
a convergent 
Bockstein Spectral Sequence that begins with $E(2)^*(X)$  
and where there can only be differentials $d_1$ through $d_7$.

We have used two versions of this spectral sequence in
the past.  
Here we use the  untruncated
version that converges to zero.  

We give a simplified 
summary of the 
Bockstein Spectral Sequence 
(BSS) for computing $ER(2)^*(X)$ from $E(2)^*(X)$.

\begin{thm}[{\cite{NituP}[Theorem 4.2]}] \ 
\begin{enumerate}
\item
The exact couple gives a spectral sequence, $E_r$,
of $ER(2)^*$ modules,
starting with
$$
E_1 \simeq E(2)^*(X) \quad \text{and ending with} \quad 
E_{8} = 0.
$$
\item
$
d_1(y) = v_2^{-3}(1-c)(y)$
where
$c(v_i)= -v_i
$
and $c$ comes from complex conjugation.
\item
The degree of $d_r$ is $17r +1$.
\item
The targets of the 
$d_r$ represent the $x^r$-torsion generators of $ER(2)^*(X)$. 
\end{enumerate}
\end{thm}

\begin{defn}
Let $K_i$ be the kernel of $x^i$ on $ER(2)^*(X)$ and let
$M_i$ be the image of $K_i$ in $ER(2)^*(X)/(xER(2)^*(X)) \subset E(2)^*(X)$.
We call $M_r/M_{r-1} \simeq$ image $d_r$ the {\bf $x^r$-torsion generators.}
\end{defn}

\begin{remark}
All of our BSSs in this paper have only even degree elements, so we
always have $d_2 = d_4 = d_6 = 0$.
In fact, $d_5$ never shows up here although we have seen it with other even degree spaces.
\end{remark}

\begin{remark}[The BSS on the coefficients.]
\label{coef}
For our purposes, it is important to know how this works
for the cohomology of a point (\cite[Theorem 3.1]{Kitch-Wil-BO}).  
The differential $d_1$ is on $E(2)^* = \Zp[\vhat_1,v_2^{\pm 1}] $,
which can now be rewritten as $\Zp[\vhat_1]\{v_2^{0-7}\}$.
The differential, $d_1$,
commutes with $\vhat_1$ and $v_2^2$ so 
all that matters here is 
$d_1(v_2) = 2 v_2^{-2}$.

The $E_2$ term becomes $\Zq [\vhat_1]\{v_2^{0,2,4,6}\}$.
We have $d_3$ commutes with
$\vhat_1$ and $v_2^4$,  and $d_3(v_2^2) = \vhat_1 v_2^{-4}$.

This leaves us with only $\Zq\{v_2^{0,4}\}$.
We have $d_7$ commutes with
$v_2^8 = \vhat_2^{-1} = 1$ 
and $d_7(v_2^4) = \vhat_2 v_2^{-8}= \vhat_2^2 = v_2^{-16} = 1$,
so $E_8 = 0$.
\end{remark}

Using this approach to
$ER(2)^*$ 
we see
that the $x^1$-torsion is generated by $\Zp [ \vhat_1 ] \{ 2v_2^{0,2,4,6} \}$,
the $x^3$-torsion by $\Zq[\vhat_1]\{\vhat_1 v_2^{0,4}\}$, 
and the $x^7$-torsion
by $\Zq $.  The previous description of $ER(2)^*$ is easy
to relate to this now.  The $x$-torsion is given by
$\Zp[\vhat_1]$ on the $\alpha_i$, $0 \le i < 4$.   The 
$x^3$-torsion is generated over 
$\Zq[\vhat_1]$ 
on $\vhat_1= \alpha$ and $w$.
Finally, the $x^7$-torsion is given by $\Zq$.

The complex conjugate of the BSS comes from $E(2)$, but Lorman
shows in
\cite[Lemma 4.1]{Vitaly}
that the complex conjugate of $\uhat \in E(2)^{-16}(\cp)$, $c(\uhat)$,
can be calculated using the formal group law for $E(2)$ from
$\hat{F}(\uhat,c(\uhat)) = 0$.

\begin{remark}
The standard formal group law for $E(2)$ is $F(x,y)$ with the degrees
of $x$ and $y$ equal to two.  The element $F(x,y)$ also has degree two.
Let $\hat{x} = v_2^3 x$ and $\hat{y} = v_2^3 y$.  Replace $v_i$ in
$F$ with $\vhat_i$.  This gives us $\hat{F}(\hat{x},\hat{y}) =
v_2^3 F(x,y)$ of degree $-16$.  
\end{remark}

We need some basic easily computed formulas, which we just quote
here.  We use Araki's generators.
These are all modulo $x^i y^j$, $i+j > 4$ or $\uhat^5$.
\[
\begin{split}
\hat{F}(\hat{x},\hat{y})&  = \hat{x}+\hat{y} + \vhat_1 \hat{x}\hat{y} + 
\vhat_1^2 (\hat{x}^2 \hat{y} + \hat{x} \hat{y}^2) \\
&
+ 
(\frac{
\text{\tiny{6}}
}{\text{\tiny{7}}} \vhat_1^3 + 
\frac{
\text{\tiny{2}}
}{
\text{\tiny{7}}
} \vhat_2)
(\hat{x}^3 \hat{y}+
\hat{x} \hat{y}^3) +
(\frac{
\text{\tiny{16}}
}{
\text{\tiny{7}}
} \vhat_1^3 + \frac{
\text{\tiny{3}}
}{
\text{\tiny{7}}
}\vhat_2) \hat{x}^2 \hat{y}^2 \\
c(\uhat) & = 
- \uhat + \vhat_1 \uhat^2 
- \vhat_1^2 \uhat^3 + 
(\frac{
\text{\tiny{10}}
}{
\text{\tiny{7}}
} \vhat_1^3 + \frac{
\text{\tiny{1}}
}{
\text{\tiny{7}}
}\vhat_2) \uhat^4 
\end{split}
\]
We collect the basics we need:

\begin{lemma} 
\label{lemma}
\[
\begin{array}{ccccl}
	c(\uhat)& = & -\uhat + \vhat_1 \uhat^2 & \hspace{.05in} & \mod (\uhat^3) \\
	c(\uhat)& = & \uhat + \vhat_1 \uhat^2 + \vhat_1^2 \uhat^3
+\vhat_2 \uhat^4 
& \hspace{.05in}  &
\mod (2, \uhat^5) \\
\pntr  =  \uhat c(\uhat) &= &- \uhat^2 
& \hspace{.05in}  &
\mod (\uhat^3) \\
\end{array}
\]
where 
$\pntr \in ER(2)^{-32}(\C \PP^\infty)$ 
maps to
$\uhat\, c(\uhat) \in E(2)^{-32}(\C \PP^\infty)$
and
is 
a modified 
first
Pontryagin class.
\end{lemma}

\begin{proof}
	This all follows from the preceding formulas.
\end{proof}

Recall that 
$$
d_1(y) = v_2^{-3}(1-c)(y).
$$
We rewrite some of our basic facts from lemma \ref{lemma} in our present
terminology keeping in mind that in $E(2)^*(-)$, $\vhat_2 = 1 = v_2^{-8}$ and 
$\pntr = \uhat \, c(\uhat)$.

\begin{lemma}\label{lemma2}
\[
\begin{array}{cclcl}
c(\uhat) & = & -\uhat - \vhat_1 \pntr & \hspace{.05in} & \mod (\pntr \uhat ) \\ 
c(\vhat_1) & =  &\vhat_1 & & \\
c(v_2) & = & -v_2 & & \\
d_1(\uhat) &= & 2 v_2^{-3} \uhat & \hspace{.05in} & \mod (\pntr) \\
d_1(v_2 \uhat) &= & 0 & \hspace{.05in} & \mod (\pntr) \\
d_1(\uhat) &=  & v_2^{-3} \vhat_1 \pntr & \hspace{.05in} & \mod (2, \pntr \uhat) \\
d_1(\uhat) &=  & v_2^{-3}( \vhat_1 \pntr + \vhat_1^3 \pntr^2 + \pntr^2) & 
\hspace{.05in} & \mod (2, \pntr^2 \uhat) \\
d_1(v_2) &= & 2 v_2^{-2} & \hspace{.05in} &  \\
d_1(v_2 \pntr) &= & 2 v_2^{-2} \pntr & \hspace{.05in} & \mod ( \pntr \uhat) \\
\end{array}
\]
\end{lemma}
\begin{proof}
There is one minor new thing here, the formula for $d_1(\uhat)$ mod $(2,\pntr^2 \uhat)$.
We do have $d_1(\uhat) = v_2^{-3}(\vhat_1 \uhat^2 + \vhat_1^2 \uhat^3 + \uhat^4)$ and
$\pntr = \uhat \, c(\uhat) = \uhat (\uhat + \vhat_1 \uhat^2 +\vhat_1^2 \uhat^3 )
= \uhat^2 + \vhat_1 \uhat^3 + \vhat_1^2 \uhat^4$.
Replace the $\uhat^2$ with $\pntr + \vhat_1 \uhat^3 + \vhat_1^2 \uhat^4$ to get
$d_1(\uhat) = v_2^{-3}(\vhat_1 \pntr + \vhat_1^2 \uhat^3 + \vhat_1^3 \uhat^4 
+ \vhat_1^2 \uhat^3 + \uhat^4)$.  Two the terms cancel out and, modulo higher terms,
$\uhat^4 = \pntr^2$.
\end{proof}

\section{A sketch of the approach}
\label{sketch}

The Bockstein spectral sequence for a general space $X$,  $E(2)^*(X)$ to $ER(2)^*(X)$,
concludes with $E_8 = 0$.  In the two cases of interest to us,
namely, $\Smash^n \cp$ and $MU(n)$, the spectral sequence is even
degree.  In fact, the only differentials are $d_1$, $d_3$, and $d_7$.
The last two are quite easy to do once $d_1$ has been computed. 
Although $d_1$ is complicated, we have an explicit algebraic formula for it.
We require a spectral sequence to compute $d_1$ though.

In this section, we sketch how we approach the computation of $d_1$
in our two cases.  We do this here without inserting the necessary 
technical details in hopes of clarifying our computations. 
When it comes time to actually do the computations, we can adjust what
we present here to be rigorous, and, in the process, add the gruesome 
technical details.

Our general description begins with an $E_1$ similar to the following:
$$
R\{ v_2^{0-7} \uhat^\ep\} \; \text{with} \; \uhat^\ep = 
\uhat_1^{\ep_1} 
\uhat_2^{\ep_2} 
\ldots
\uhat_n^{\ep_n} 
\quad \ep_k \le 1.
$$
Our $R$ has no torsion and $d_1$ commutes with $R$ and $v_2^2$. 

\begin{defn}
	\label{wt}
Define $W_j$ to be the set of $\uhat^\ep$ with $i_k = 0$ for $k < j$ and
$i_j = 1$.
We also include $W_{n+1}$ with all $i_k = 0$.
\end{defn}

So, we are dealing with something of the form
$$
R\{
v_2^{0-7} W_1,
v_2^{0-7} W_2,
\ldots,
v_2^{0-7} W_{n},
v_2^{0-7} W_{n+1}
\}
$$

We want to describe how our computations will go.
We put a filtration on 
$R\{ v_2^{0-7} \uhat^\ep\}$ and use it to produce a spectral sequence
for computing $d_1$.  
More on this later, but for now we do not need the details.

First, we have a differential $d_{1,0}$.  
We need a new definition:
$$
s(\ep) = \sum \ep_k.
$$

Our $d_{1,0}$
kills off lots of elements and 2. (Mod higher filtrations.)
$$
d_{1,0}(\uhat^\epsilon) 
= v_2^{-3} (\uhat^\epsilon - c(\uhat^\epsilon) )
$$
$$
= v_2^{-3} (\uhat^\epsilon - \prod_{\ep_k=1} c(\uhat_k) )
= v_2^{-3} (\uhat^\epsilon - (-1)^{s(\ep)} \uhat^\ep)
$$
So,
$$
\begin{array}{lcl}
d_{1,0}( \uhat^\epsilon) = 2 v_2^{-3} \uhat^\epsilon
&
\quad & s(\epsilon) \text{ odd}\\
d_{1,0}( \uhat^\epsilon) = 0 &
\quad &  s(\epsilon) \text{ even}
\end{array}
$$
With the $v_2$ in front, knowing $c(v_2) = - v_2$, we get
$$
\begin{array}{lcl}
d_{1,0}(v_2 \uhat^\epsilon) = 2 v_2^{-2} \uhat^\epsilon
&
\quad & s(\epsilon) \text{ even} \\
d_{1,0}(v_2 \uhat^\epsilon) = 0
& \quad & s(\epsilon) \text{ odd}
\end{array}
$$
To describe this answer though, we need some new notation:
$$
\begin{array}{lcl}
\voe  =  v_2 & \quad & s(\epsilon) \text{ odd} \\
\voe  =  1 & \quad & s(\epsilon) \text{ even} \\
\end{array}
$$

The end result of this $d_{1,0}$ is 
$$
E_{1,1} = 
R/(2)\{\voe v_2^{0,2,4,6} \uhat^\ep\}
$$
$$
=
R/(2)\{
\voe v_2^{0,2,4,6} W_1,
\voe v_2^{0,2,4,6} W_2,
\ldots,
\voe v_2^{0,2,4,6} W_{n},
v_2^{0,2,4,6} W_{n+1}
\}
$$
Note the lack of $\voe$ for $W_{n+1}$ because $s(\ep) = 0$ when all $\ep_k = 0$.

To compute the rest of $d_1$ we will inductively compute it on the $W_j$, calling
this $d_{1,j}$.  As it turns out, each $d_{1,j}$ is injective, so when
we have computed $d_{1,j}$ for all $ 0 < j \le n$, all we have left is a quotient
of
$$
R/(2)\{
v_2^{0,2,4,6} W_{n+1}
\}
$$
and this will be our associated graded object for $E_2$ of the Bockstein spectral sequence.

After $d_{1,0}$, we are always working mod (2) and in our cases we always have, roughly
speaking,
$c(\uhat_j) = \uhat_j + h_j$, with $h_j \in R/(2)$ in our as yet undefined
associated graded object from our as yet undefined filtration. 

We restrict our $d_1$, called $d_{1,1}$, on $E_{1,1}$ to $W_1$, i.e. elements with
$\ep_1 = 1$.
$$
d_{1}(\voe \uhat^\epsilon) 
=\voe v_2^{-3} (\uhat^\epsilon + c(\uhat^\epsilon) ) =
$$
$$
\voe v_2^{-3} (\uhat^\epsilon + \prod_{\ep_j=1} c(\uhat_j) )
=\voe v_2^{-3} (\uhat^\epsilon + \prod_{\ep_j=1} (\uhat_j+h_j ))
$$
The $\uhat^\ep$ cancel out and we must have at least one $h_j$, and when we
have it, we do not have the $\uhat_j$, so 
the contribution from $\uhat_j$ raises our (unnamed) filtration.
We show, in our cases, the the contribution from $j > 1$ raises the filtration
more than for $j = 1$ so we only need to look at the action on $\uhat_1$.

This becomes
$$
=\voe v_2^{-3} (h_1 \uhat^{\ep - \Delta_1})
$$
where $\Delta_j$ has a 1 in the $j$-th coordinate and zeros elsewhere. 
When $\ep_1 = 1$, we show that this is injective so our $E_{1,2}$
will be a quotient of
$$
R/(2)\{
\voe v_2^{0,2,4,6} W_2,
\ldots,
\voe v_2^{0,2,4,6} W_{n},
v_2^{0,2,4,6} W_{n+1}
\}
$$

This process continues with each $d_{1,j}$ injective giving us $E_{1,j+1}$ a
quotient of
$$
R/(2)\{
\voe v_2^{0,2,4,6} W_{j+1},
\ldots,
\voe v_2^{0,2,4,6} W_{n},
v_2^{0,2,4,6} W_{n+1}
\}
$$

\section{The  Filtration}
\label{secfilt}

By complex orientability, we have
$$
E(2)^*(\Smash^n \cp) \iso E(2)^*[[\uhat_1,\uhat_2,\ldots,\uhat_n]]
\{\uhat_1 \uhat_2 \cdots \uhat_n \}
$$
The class $\uhat \, c(\uhat)$ coming from $\pntr$ is a permanent cycle.
Let $\pntr_i$ be the class associated with the $i$-th copy of $\cp$
in our smash product.

Because $\pntr = - \uhat^2$ mod higher powers, we can replace
our description of $E(2)^*(\Smash^n \cp)$.  We need some notation
first.
\[
\begin{array}{lcl}
I = (i_1, i_2, \ldots, i_n)  & s(I) = \sum i_k & i_k \ge 0 \\
\epsilon = (\epsilon_1, \epsilon_2, \ldots, \epsilon_n)  
&
 s(\epsilon) = \sum \epsilon_k
& \ep_k = 0 \text{ or } 1
\end{array}
\]
Define
$$
\pntr^I \uhat^\epsilon =
\pntr_1^{i_1} \uhat_1^{\epsilon_1}
\pntr_2^{i_2} \uhat_2^{\epsilon_2}
\cdots
\pntr_n^{i_n} \uhat_n^{\epsilon_n}
$$
We have 
\[
E(2)^*(\Smash^n \cp) \iso E(2)^*\{ \pntr^I \uhat^\epsilon \}
\iso
\Zp[\vhat_1]\{ v_2^{0-7}\pntr^I \uhat^\epsilon \} 
\qquad i_k + \epsilon_k > 0
\]
We need to put our filtration on this.

\begin{defn}
	\label{order}
We put an order on the pairs $(I,\epsilon)$ as follows.
If the length of $(I',\ep')$, $\ell(I',\ep') = 2s(I')+s(\epsilon') > 2 s(I) + s(\epsilon)$
then 
$(I',\epsilon') > (I,\epsilon)$.
If $2s(I')+s(\epsilon') = 2 s(I) + s(\epsilon)$, 
$2{i'}_j + {\epsilon'}_j = 
2{i}_j + \epsilon_j$ 
for $k < j \le n$, and
$2{i'}_k + {\epsilon'}_k >
2{i}_k + \epsilon_k$ 
then 
$(I',\epsilon') > (I,\epsilon)$.

The $(I,\epsilon)$ now form an ordered set and we can use
them to give a filtration on $E(2)^*(\Smash^n \cp)$ 
as follows:
$$
F(I,\epsilon) = 
\Zp[\vhat_1]\{ v_2^{0-7}
\pntr^{I'} \uhat^{\epsilon' }
\}
\qquad  (I',\epsilon') > (I,\epsilon)
$$
\end{defn}

The associated graded object still looks the same:
$$
E_{1,0}(I,\epsilon)=E(2)^*(\Smash^n \cp) \iso
\Zp[\vhat_1]\{ v_2^{0-7}
\pntr^I \uhat^\epsilon 
\}
\qquad i_k + \epsilon_k > 0
$$
Note that in degrees $16*$, this is just $\Zp[\vhat_1]
\{\pntr^I \uhat^\epsilon \} $, 
$i_k + \epsilon_k > 0$.

In general, we will suppress the $(I,\ep)$ notation associated with
this filtration.  We will use it, but the associated graded
object will be implicit, not explicit.  A certain amount of clutter
is avoided without loss, we hope, of clarity.

\section{Computing $d_{1,0}$ for $\Smash^n \cp$}
\label{d10}

The setup of our computation in section \ref{sketch} now applies.
The zeroth differential is computed there giving us:

\begin{prop}
	After computing $d_{1,0}$ for $\Smash^n \cp$, we get
$$
E_{1,1} \iso 
\Zq[\vhat_1 ]\{\voe v_2^{0,2,4,6} \pntr^I \uhat^\epsilon \}
$$
with $i_k + \epsilon_k > 0$.
The $x^1$-torsion generators  detected by $d_{1,0}$ are represented by:
$$
\Zp[\vhat_1]\{2 \voe  v_2^{0,2,4,6} \pntr^I \uhat^\epsilon \}
=
\Zp[\vhat_1]\{ \voe  \alpha_i \pntr^I \uhat^\epsilon \}
$$
\end{prop}

\section{Computing $d_{1,1}$ for $\Smash^n \cp$}
\label{d11}

After $d_{1,0}$, we are working mod $(2)$.
Following section \ref{sketch}, we start our computation of $d_1$ on elements
with $\ep_1 = 1$.
The main formula we need now is:  $c(\uhat) = \uhat + \vhat_1 \pntr$ modulo
$(2, \pntr \uhat)$, where we are now invoking the filtration and
looking only at the representative in the associated graded object.
We call the map restricted to the $\uhat^\ep$ with $\ep_1 = 1$, $d_{1,1}$.
Our description of what happens is exactly what we developed in section \ref{sketch}.

Our $d_1$ commutes with $v_2^2$, $\pntr_i$ and $\voe$.
\[
\begin{split}
d_{1,1}(\voe \uhat^\epsilon) 
&= v_2^{-3}\voe (\uhat^\epsilon + c(\uhat^\epsilon) ) \\
& = v_2^{-3}\voe (\uhat^\epsilon + \prod_{\ep_k = 1} (\uhat_k + \vhat_1 \pntr_k)\uhat^{\ep - \Delta_k})\\
\end{split}
\]
The $\uhat^\ep$ cancels out and we always have some $\vhat_1 \pntr_k$.  If we use more than one of these,
the length of $(I,\ep)$ goes up higher than if we just use one, and so we can ignore those terms
in the spectral sequence.  In the case of any single $k > 1$, the filtration is already higher than
for $k = 1$. Since we are only looking at terms with $\ep_1 = 1$, modulo higher filtrations, 
this is just:
\[
=  v_2^{-3}\voe    \vhat_1 \pntr_1 \uhat^{\ep - \Delta_1}
\]

\begin{prop}
	After computing $d_{1,1}$ for $\Smash^n \cp$, we get
$$
E_{1,2} \iso 
\Zq\{\voe v_2^{0,2,4,6} \pntr^I \uhat^\epsilon \}
\quad \ep_1 = 0
$$
with $i_k + \epsilon_k > 0$.
The $x^1$-torsion generators detected by $d_{1,1}$ are represented by:
$$
\Zq[\vhat_1 ]\{ \voe v_2^{0,2,4,6}  \vhat_1 \pntr^I \uhat^\epsilon \}
\quad \ep_1 = 0
$$
\end{prop}

Note that because we are in the smash product, $\ep_1 = 0$
implies that $i_1 > 0$.

As mentioned in section \ref{sketch}, $d_{1,1}$ is injective on terms with
$\ep_1 = 1$, so all remaining terms have $\ep_1 = 0$.

It might be premature to discuss such things, but the above
is consistent with the results for $ER(2)^*(\cp)$ from
\cite[Theorems 3.1 and 4.1]{Kitch-Lor-Wil-CPn}, i.e. the $n=1$ case,
even if, at first glance, they don't look the same.

\section{Computing $d_{1,j}$ for $\Smash^n \cp$}

By induction, when we start our work with $d_{1,j}$, 
we find that all we have left
are $\uhat^\ep$ with $i_1 = i_2 = \cdots = i_{j-1} = 0$.

We are still working mod 2.
The main new formula used in this section is equation (4.2) from \cite{Kitch-Lor-Wil-CPn}:
\begin{equation}
	\label{oldformula}
0 = \vhat_1 \uhat^2 + \vhat_2 \uhat^4 = \vhat_1 \pntr + \pntr^2
\qquad \mod (\pntr^2 \uhat)
\end{equation}
This formula requires a context.  It doesn't just apply anywhere.  We always
start out with $\vhat_1$ being non-zero and this formula can't be applied.  
This formula comes about because the image of $d_1$ is zero in $E_2$. 
Consequently, the formula only applies when $d_1$ (in the spectral sequence
for $d_1$ that is) comes along and ``kills'' the $\vhat_1$, which isn't really
dead in the sense that the formula above can tell us what it really is.
As it turns out, we inadvertently reprove the formula in \ref{oldformula2},
so if this explanation isn't satisfactory, you can find the details there.

\begin{prop}
	\label{propj}
	After computing $d_{1,j}$ for $\Smash^n \cp$, we get
$$
E_{1,j+1} \iso 
\Zq\{\voe v_2^{0,2,4,6} \pntr^I \uhat^\epsilon \}
$$
for $1 < j \le n$. We have $i_k + \epsilon_k > 0$. 
$$
\ep_k = 0\  \text{ for } k \le j,
\quad i_k = 1\  \text{ for } k < j 
$$
The $x^1$-torsion generators detected by $d_{1,j}$ are represented by:
$$
\Zq\{\voe v_2^{0,2,4,6} \pntr^I \uhat^\epsilon \}
\quad \ep_k = 0\  \text{ for } k \le j
$$
$$
i_k = 1\  \text{ for } k < j -1,
\quad i_{j-1} > 1  
$$
We have 
$E_{1,n+1} = E_{1,\infty}$, which is our associated graded object for 
the BSS
$E_2$ for computing $ER(2)^*(\Smash^n \cp)$ from $E(2)^*(\Smash^n \cp)$, is
$$
\qquad \Zq\{v_2^{0,2,4,6} \pntr^I  \}
\qquad i_k = 1 \  \text{ for } k < n 
$$
or
$$
\qquad \Zq\{v_2^{0,2,4,6} \pntr_1 \pntr_2 \ldots \pntr_{n-1}\pntr_n^{i_n}  \}.
$$
\end{prop}
Note that 
if all $\ep_k = 0$, $s(\ep)$ is even.

\begin{proof}
We have already computed $E_{1,2}$, so our induction is started. 
Assume we have computed $E_{1,j'+1}$ and $d_{1,j'}$ for $j' < j$.
We need to compute $d_{1,j}$ on $E_{1,j}$ to get $E_{1,j+1}$.
We compute $d_{1,j}$ only on those $\uhat^\ep$ with $i_j = 1$.

We use our filtration to get $d_1(\uhat)=v_2^{-3} \vhat_1 \pntr$ from lemma \ref{lemma2}.
As in the case of $j=1$,
when $\ep_j = 1$, $d_1$ applied to a $\uhat_k$, with $k > j$, increases the filtration
more than $d_1$ applied to $\uhat_j$ does.
This gives:
$$
d_{1,j} (\voe v_2^{0,2,4,6} \pntr^I \uhat^\ep) 
= 
v_2^{-3} \vhat_1 \voe v_2^{0,2,4,6} \pntr^{I+\Delta_j} \uhat^{\ep-\Delta_j}
$$
Unfortunately, $\vhat_1$ doesn't show up in in the associated
graded object for $E_{1,j}$ so we need to
find an equivalent element that represents this.
Note that $i_k = 1$ for $k < j-1$, and $i_{j-1} > 0$. 
We use formula 
\ref{oldformula},  
$\vhat_1 \pntr = \pntr^2$.  The lowest $i$ with $\pntr_i^2 \neq 0$ in $E_{1,j}$
is $i = j-1$, so, this term is, mod higher filtrations, represented by:
$$
v_2^{-3} \voe v_2^{0,2,4,6} \pntr^{I+\Delta_j + \Delta_{j-1}} \uhat^{\ep-\Delta_j}.
$$
The result follows.
\end{proof}

\begin{remark}
At this stage, we are done with $d_1$ for degree reasons and have computed
$E_2$ for theorem \ref{smash}.
\end{remark}

\section{Summary of the $x^1$-torsion generators for $ER(2)^*(\Smash^n \cp)$}

We just collect from the previous sections:

\begin{thm}
Representatives for the $x^1$-torsion generators in our associated graded object for
$ER(2)^*(\Smash^n \cp)$ are given by:
$$
\Zp[\vhat_1]\{ \voe  \alpha_i \pntr^I \uhat^\epsilon \}
\quad
0 \le i < 4
$$
$$
\Zq[\vhat_1 ]\{ \voe v_2^{0,2,4,6}  \vhat_1 \pntr^I \uhat^\epsilon \}
\quad \ep_1 = 0
$$
For $1 < j \le n$,
$$
\Zq\{\voe v_2^{0,2,4,6} \pntr^I \uhat^\epsilon \}
\quad \ep_k = 0\  \text{ for } k \le j
$$
$$
i_k = 1\  \text{ for } k < j -1,
\quad i_{j-1} > 1  
$$
\end{thm}

\section{Computing $d_3$ and $d_7$ for $\Smash^n \cp$}

We have finished our computation of $d_1$ and we get  $E_2 = E_3$. 

\begin{prop}
	Our associated graded version of the BSS $E_4$
 for computing $ER(2)^*(\Smash^n \cp)$ from $E(2)^*(\Smash^n \cp)$ is
$$
E_4 = E_5 = E_6 = E_7 = 
\Zq\{v_2^{0,4} \pntr^I \} 
\qquad i_k = 1 
$$
The $x^3$-torsion generators are represented by
$$
\Zq\{v_2^{0,4} \pntr^I \} 
\qquad i_k = 1\  \text{ for } k < n 
\qquad i_n > 1 
$$
\end{prop}

\begin{remark}
	By definition, all $x^i$-torsion generators inject into $E(2)^*(-)$.
In particular, the $x^1$-torsion generators (all are of even degree) inject.
The $x^3$-torsion generators are all in degrees $8*$.
	The degree of $x$ is -1 mod (8) so for $x^3$-torsion, we only have
	$x$ and $x^2$ times elements in degree $8*$.  
Consequently, all of the elements in degrees $4*$ that
	we have studied so far inject.   Lots of elements have
$x^2$ times them non-zero, so there are 
many elements in degrees $ -2$ mod (8)
that don't inject.
\end{remark}

\begin{proof}
For degree reasons there is no more to $d_1$ and there is no $d_2$.
All of the $\pntr_k$ are permanent cycles (\cite[Proposition 5.1]{Vitaly}) and $d_3$ commutes with
$v_2^4$,
so the computation
of $d_3$ is based on:
$$
d_3(v_2^2) = \vhat_1 v_2^{-4}
$$
from the action on the coefficients.
We have the relation:
$$
\vhat_1 \pntr_n = \pntr_n^2
$$
Because $i_k = 1$ for $k < n$, we have
$$
d_3(v_2^2 \pntr^I)
=
v_2^{-4} \vhat_1 \pntr^I
=
v_2^{-4} \pntr^{I+\Delta_n}
$$
From this we get our $E_4 = E_5 = E_6 = E_7$ (for degree reasons). 
\end{proof}

Starting with our $E_7$ and recalling
from the coefficients that
$$
d_7(v_2^4) = 1,
$$
we get:

\begin{prop}
	For the BSS 
 for computing $ER(2)^*(\Smash^n \cp)$ from $E(2)^*(\Smash^n \cp)$, we have 
	 $E_8 = 0$.
	The $x^7$-torsion generator is 
	$$
	\Zq\{\pntr_1 \pntr_2 \ldots \pntr_n \} 
	$$
This generator is in degree $16*$ ($-32n=16n$).
\end{prop}

\begin{remark}
	The only element divisible by $x$ in degree 4 mod (8) is
	$x^4 \pntr_1 \pntr_2 \ldots \pntr_n$.
	Consequently, it is the only element in the kernel of the map
	in degrees $4*$.  
Similarly, 	
	$x^6 \pntr_1 \pntr_2 \ldots \pntr_n$
	is the only element in degree $8*-6$ divisible by $x$.
	This concludes the proof of theorem \ref{smash}
	and part of theorem \ref{inj}, and the remark that follows.
	If you want particularly clean statements, stick with elements
	in degree $16*$.  
	In all our statements, just require $s(\ep)$ to be even and ignore
	the $v_2^{2,4,6}$.
	Historically, those are the only elements
	that have mattered to us, but it takes so little effort to 
	get the injection for $8*$, it seems obligatory.
	Here we still require $s(\ep)$ to be even, but we only ignore $v_2^{2,6}$,
	leaving $v_2^{0,4}$.
\end{remark}

\section{Preliminaries for $BU(n)$}
\label{startbu}

Because of
the stable splitting, $BU(n) = MU(n)\vee BU(n-1)$, 
\cite{MitPrid},
we can compute
$ER(2)^*(MU(n))$   
instead of $ER(2)^*(BU(n))$. 

So, rather than study the map
$$
{\sf X}^n \cp \lra BU(n)
$$
we will mainly look at:
$$
\Smash^n \cp \lra MU(n)
$$

Because $E(2)$ is a complex orientable theory, we have
the usual
$$
E(2)^*(BU(n)) \iso E(2)^*[[c_1,c_2,\ldots,c_n]]
$$
where the $c_k$ are the generalized Conner-Floyd Chern classes.
To see $E(2)^*(MU(n))$, we just look at the ideal generated
by $c_n$.  So, we have:
$$
E(2)^*(MU(n)) \iso E(2)^*[[c_1,c_2,\ldots,c_n]]\{c_n\}
$$
We need to `hat' these Chern classes just as we did with
$\uhat$ for $\cp$.  Define (keeping in mind that $v_2$ is a unit):
$$
\chat_k = v_2^{3k} c_k.
$$
This puts $\chat_k$ in degree $2k - 18k = -16k$.

We need to use the well-known fact that for complex oriented
theories, $G^*(BU(n))$ injects into 
$G^*( {\sf X}^n \cp)$.
Each $c_k$, or, respectively, $\chat_k$, goes to the $k$-th symmetric
function on the $u_i$, respectively, $\uhat_i$.
Similarly for the the map of the smash product to $MU(n)$.  Here, we
have $\chat_n$ goes to $\uhat_1 \uhat_2 \cdots \uhat_n$.

For $J = (j_1,j_2,\ldots,j_n)$, let
$$
\chat^J = \chat_1^{\,j_1} \chat_2^{\,j_2} \cdots \chat_n^{\,j_n}.
$$
We can write
$$
E(2)^*(MU(n)) \iso E(2)^*\{\chat^J \}
\qquad j_n > 0.
$$
We can view 
$$
E(2)^*(MU(n)) \subset E(2)^*(\Smash^n \cp)
$$
and we know how to write elements of $E(2)^*(\Smash^n \cp)$ as
$$
\pntr^I \uhat^\epsilon =
\pntr_1^{i_1} \uhat_1^{\epsilon_1}
\pntr_2^{i_2} \uhat_2^{\epsilon_2}
\cdots
\pntr_n^{i_n} \uhat_n^{\epsilon_n}
$$
Every element $z \in E(2)^*(\Smash^n \cp)$ can be written as a sum of such elements
(with coefficients).  These elements are ordered using the order on $(I,\ep)$
from \ref{order}.
\begin{defn}
	The {\bf lead term} of $z \in E(2)^*(\Smash^n \cp)$ is the term of lowest order.
\end{defn}
The lead term of any symmetric function must be of the form
$\pntr^I \uhat^\epsilon$ 
with
$$
2i_1 + \ep_1 \ge \cdots \ge
2i_k + \ep_k \ge 2i_{k+1} + \ep_{k+1} \ge \ldots \ge 2i_n + \ep_n > 0
$$
\begin{defn}
We call this {\bf property  A} and use it constantly from here on, but
without having to repeat the above often.
\end{defn}
Although many symmetric functions could have the same lead term, given
a 
$\pntr^I \uhat^\epsilon$ 
with \A, we can construct a unique symmetric function, $w_{I,\ep}$,
with this as its lead term.
Our $\ww$ is just the sum of all distinct permutations of our
$\pntr^I \uhat^\epsilon$, keeping in mind that the $\pntr_i$ and the $\uhat_i$ move together.
These symmetric functions $\ww$ generate $E(2)^*(MU(n)) \subset E(2)^*(\Smash^n \cp)$.
Our computations will take place entirely in this image.

Consider
$$
E(2)^*(MU(n)) \iso  \Zp[\vhat_1]
\{
	v_2^{0-7} w_{I,\ep} 
\}
\subset E(2)^*(\Smash^n \cp).
$$
For our filtration,
we just use the order, \ref{order}, on the $(I,\ep)$,
which all have \A.  This is
the same as using it on the lead term.

Recall that $d_1$ commutes with $\pntr_i$, and $v_2^2$.

Similar to the computation in section \ref{d10}, we can compute
$d_{1,0}$ (modulo higher filtration) on every term of $w_{I,\ep}$ to get
\[
	\begin{split}
		d_{1,0}(w_{I,\ep}) & = 2 v_2^{-3} w_{I,\ep}
 \quad  s(\ep) \text{ odd } \\
 d_{1,0}(v_2 w_{I,\ep})  & = 2 v_2^{-2} w_{I,\ep}
 \quad  s(\ep) \text{ even }
 \end{split}
 \]

\begin{prop}
	With \A, after computing $d_{1,0}$ for $MU(n)$, we have:
$$
E_{1,1} \iso 
\Zq[\vhat_1 ]\{\voe v_2^{0,2,4,6} w_{I,\ep} \}
$$
The $x^1$-torsion generators detected by $d_{1,0}$ are represented by:
$$
\Zp[\vhat_1]\{2 \voe  v_2^{0,2,4,6} w_{I,\ep} \} = 
\Zp[\vhat_1]\{ \voe  \alpha_i w_{I,\ep} \}
$$
\end{prop}

\section{Different descriptions of $E(2)^*(MU(n))$}
\label{transition}

We find it easiest to make our computations with the $\ww \in E(2)^*(\Smash^n \cp)$,
but it would be more traditional to think in terms of Chern classes in $E(2)^*(MU(n))$.
So, we now show how to
relate the $\chat^J$ to the $w_{I,\ep}$.

In the product, the image of $\chat_k$ is the $k$-th symmetric function on the
$\uhat_i$.  The lead term in the sum that makes up the
symmetric function is:
$$
\uhat(k) = \uhat_1 \uhat_2 \cdots \uhat_k.
$$

Modulo higher terms in the filtration, we have $\uhat^2 = - \uhat\, (- \uhat)
= - \uhat \, c(\uhat) = - \pntr$, so,
in the smash product, 
the lead term of the image of $\chat_k \chat_n$ 
is (modulo higher terms):
$$
\uhat(k)\uhat(n)=
\uhat_1^2 \uhat_2^2 \cdots \uhat_k^2 \uhat_{k+1} \cdots \uhat_n =
(-1)^k \pntr_1 \pntr_2 \cdots \pntr_k \uhat_{k+1} \uhat_{k+2} \cdots \uhat_n
$$

We have the lead term of the image of $c^J$ 
$$
\uhat(1)^{j_1}
\uhat(2)^{j_2}
\cdots
\uhat(n)^{j_n}
\lra
\uhat_1^{\sum_{i=1}^n  j_i}
\uhat_2^{\sum_{i=2}^n  j_i}
\cdots
\uhat_k^{\sum_{i=k}^n  j_i}
\cdots
\uhat_n^{ j_n}
$$

We prefer to replace all the $\chat_k^2$ with the 
Pontryagin classes $\Pntr_k$ from the introduction.
Recall, they map to $\chat_k \, c(\chat_k) \in E(2)^{-32k}(BU(n))$. 
Complex conjugation, $c$, is defined on $E(2)$.  It can be computed
for $E(2)^*(MU(n))$ by naturality since we know it on 
$E(2)^*(\cp)$.

Note for later purposes that $d_1(\Pntr_k) = 0$ because $d_1$ is determined by
the conjugation, $c$.
Also note that the lead term of the image of $\Pntr_k$ is 
$(-1)^k \uhat(k)^2 = \pntr_1 \pntr_2 \ldots \pntr_k$.

We rewrite
$$
E(2)^*(MU(n)) \iso
E(2)^*\{
\Pntr_1^{k_1} \chat_1^{\,r_1}
\Pntr_2^{k_2} \chat_2^{\,r_2}
\ldots
\Pntr_n^{k_n} \chat_n^{\,r_n}
\}
\quad 
0 < 
k_n + r_n 
\quad 
r_k \le 1
$$
or, simply as:
$$
E(2)^*(MU(n)) \iso
E(2)^*\{\Pntr^K \chat^{\,r}\}
\quad 
0 < 
k_n + r_n 
$$

Define $s_i$, $e_i$, $g_i$ and $\ep_i$ as follows: 
$$
s_i = k_i + k_{i+1} + \cdots + k_n
 \;
\text{ and }
 \;
e_i = r_i + r_{i+1} + \cdots + r_n 
=
2g_i + \ep_i
\quad
\ep_i \le 1.
$$
Let $i_j = s_j + g_j$, then,
using our injection, we have, modulo higher filtration: 
$$
\Pntr^K \chat^{\,r} \; \text{maps to} \; \pm w_{I,\ep}
$$
where $w_{I,\ep}$ is the symmetric function, with lead term
$$
\pntr_1^{s_1+g_1} \uhat_1^{\ep_1}
\pntr_2^{s_2+g_2} \uhat_2^{\ep_2}
\cdots
\pntr_n^{s_n+g_n} \uhat_n^{\ep_n}
=
\pntr^I \uhat^\ep.
$$
Note that, by construction, this satisfies \A.

Reversing the process to go from $\ww$ to 
$\Pntr^K \chat^{\,r}$ is unpleasant.  It is trivial to go back to  
$\pntr^I \uhat^\ep$, but that is still inside $E(2)^*(\Smash^n \cp)$.
It is best to see $\ww$ in terms of Chern classes.  Modulo higher
filtration, we have, 
\begin{equation}
\label{tochern}
\ww = \chat_1^{2i_1 + \ep_1 -2i_2 - \ep_2}
 \chat_2^{2i_2 + \ep_2 -2i_3 - \ep_3}
 \cdots
 \chat_n^{2i_n + \ep_n }.
\end{equation}
Note that when $\ep_1 = \ep_2 = 0$, we get $\chat_1^{2(i_1 - i_2)} = \pm \Pntr_1^{(i_1 - i_2)}$ mod
higher filtration.  This comes in handy later.

Although we don't need to be able to completely reverse
the process to go from $\ww$ to $\Pntr^K \chat^{\,r}$,
we do need to keep track of the parity of $s(\ep)$.

\begin{lemma}
\label{parity}
If $\ww$ is the image of $\chat^J = \chat_1^{j_1} \chat_2^{j_2}
\ldots \chat_n^{j_n}$, then the parity of $s(\ep)$ is the
same as the parity of $j_1 + j_3 + j_5 + \cdots$, or, equivalently,
$r_1 + r_3 + r_5 + \cdots$ from above.
\end{lemma}

\begin{proof}
The proof is easy.  From equation \ref{tochern}, we have $j_1 + j_3
+ j_5 + \cdots$ is, mod (2), just $\ep_1 - \ep_2 + \ep_3 - \ep_4
+\cdots$ and this has the same parity as $s(\ep)$.  Using the $r$,
we have, mod 2, $s(\ep) = r_1 + 2r_2 + 3r_3 + \cdots + nr_n$. 
Deleting all the even terms gives the same result.
\end{proof}

\section{Lemmas for our $MU(n)$ $d_1$ computations}

We are going to compute $d_1$ in the Bockstein spectral sequence using a spectral
sequence.  Our computations will be done on the image of $E(2)^*(MU(n))$ in
$E(2)^*(\Smash^n \cp)$.  This is generated by the symmetric functions 
$\ww$ where the lead term is $\pntr^I \uhat^\ep$ with
\A.  The spectral sequence we use to compute $d_1$ is based on the filtration
we have given using the ordering on the $(I,\ep)$.  Since $d_1(\ww)$ is also
a symmetric function, to compute the spectral sequence we need to know its
lead term in the associated graded object,  
i.e. the lowest filtration term of $d_1(\ww)$.
In principle, to do this, we have to compute $d_1$ on every one of the
distinct permutations that make up $\ww$.

We reduce that onerous task significantly in this section by a series of
simplifications.  First, we explain that we are working mod (2)
in a very strong sense now that we 
have computed $d_{1,0}$. 
In our actual computations, it turns out that we never need to consider
raising our filtration so much that the length of $(I,\ep)$, $\ell(I,\ep),$ is
raised by more than 3.  We don't prove that here, that just comes out
of the computations.  What we do here is show how to compute when you keep
the increase in length to less than or equal to 3.

Our differentials only act on the $\uhat$ part of $\ww$.  We show that
if $d_1$ acts on more than one $\uhat$ at a time, it increases the length
by more than 3.  The consequence of this is that we only have to take $d_1$
of one $\uhat_k$ at a time in each term of $\ww$.  That's still a lot to do,
but is already a significant simplification.

A lead term of $d_1(\ww)$ must come from $d_1$ acting on some distinct
permutation, $\pntr^J \uhat^r$,  of the lead term $\pntr^I \uhat^\ep$, and
we need only consider $d_1$ on one of the $\uhat$ in $\uhat^r$ at a time.
To be a distinct permutation other than the lead term, it cannot have \A.
For $d_1$ of it to be a lead term, $d_1$ of it must have a term with \A.  
If it doesn't have such a term with \A, then
we don't have to concern ourselves with it as it cannot be the lead term
of $d_1(\ww)$.

There are many possible distinct permutations.  Taking $d_1$ of
all of them, even using only one $\uhat$ at a time, results in a large
number of terms.  Using the considerations just discussed, 
we will
be able to eliminate from consideration almost all of them.
We reduce the relevant permutations and 
computations to a very few special cases.

That is the goal of this section.

\begin{remark}
It is important to note that we are working mod (2)
in a very strong sense.  Normally, in a spectral sequence, after
our computation of $d_{1,0}$, this would mean that
2 times an element in the associated graded object is really just an element
represented by some higher filtration term. 
However,
because $2x = 0$, we do not have such extension problems.  Two times any element is
definitely killed by $x$ and so is actually zero in the spectral sequence.
For reference, we state this as a lemma.
\end{remark}

\begin{lemma}[{\bf Two is zero}]
\label{mod2}
Two times an element in $E_{1,1}$ for $MU(n)$ is zero.  It is not represented in
the associated graded object by a non-zero element in a higher filtration.
\end{lemma}

We recall from lemma \ref{lemma2} (our long version of $d_1$):
$$
d_1(\uhat) =   v_2^{-3}( \vhat_1 \pntr + \vhat_1^3 \pntr^2 + \pntr^2) 
\quad \mod (2, \pntr^2 \uhat) 
$$
This is our main source of information for computing $d_1$ because these
are all the terms of $d_1$ we need.  

\begin{remark}[{\bf Powers of $v_2$}]
We have already introduced the notation
$\voe$.  If we apply our above $d_1$ to a $\uhat_k$, we decrease the number
of $\uhat$ in $\uhat^\ep$ by one, thus changing the parity of $s(\ep)$.  On the other hand,
the $v_2^{-3}$ changes the parity for $\voe$, so the formula $\voe \ww$
stays aligned as we do differentials.  In fact, we can generally ignore
the powers of $v_2$ when working with $d_1$ because they take care of themselves.
\end{remark}

\begin{convention}
	\label{voeconv}
Now that we have established that the $\voe$ that depends on $s(\ep)$ takes
care of itself, for the part of this section 
before our important lemmas,
we will ignore
the powers of $v_2$. 
They will be re-introduced when we get to our lemmas.
\end{convention}

\begin{defn}[{\bf Short version of $d_1$}] 
Following convention \ref{voeconv}, the {\em short version} of $d_1$ is:
\label{shortver}
$$
d_1(\uhat) =    \vhat_1 \pntr 
\quad \mod (2, \pntr \uhat) 
$$
\end{defn}
This is much of what we need, but it does run into problems that require the
long version of the formula.  When we apply this to just one $\uhat_k$ and one term of
the symmetric function, we get
$$
d_1(\pntr^J \uhat^r) = \vhat_1 \pntr^{J+\Delta_k} \uhat^{r-\Delta_k}
$$
If this element exists
and is of lowest filtration for our choice of $k$, 
we usually don't have to go further.  If there is no
$\vhat_1$ on such an element, it doesn't mean it is zero as is the case with 2.
Instead, it means that the element can be represented in a higher filtration.
Since all of our elements start off with a $\vhat_1$, if it isn't there, it means
that the short version of $d_1$ has already come along to hit it.
That doesn't make it zero, but since the image of $d_1$ is zero, it means we
have: 
$$
  \vhat_1 \pntr = \vhat_1^3 \pntr^2 + \pntr^2 
\quad \mod (2, \pntr^2 \uhat) 
$$
Always in such cases, the $\vhat_1^3\pntr^2$ isn't there as well and so belongs
in a higher filtration giving us the relation.  
\begin{relation}
\label{oldformula2}
$$
  \vhat_1 \pntr =  \pntr^2 
\quad \mod (2, \pntr^2 \uhat) 
$$
This gives a proof of equation \ref{oldformula} that we have already used 
in similar situations.
\end{relation}

This increase in filtration is significant as it involves
an increase in the length of $(I,\ep)$, $\ell(I,\ep)$.  Note that the short version
of $d_1$ increases length by one and the relation above by another 2.
We are fortunate that we never have to go beyond an increase of length 3.  Note that in the long
version of $d_1$ above, we only raise length by 3 if we need to use the $\pntr^2$
term as well.
When this happens, it is always because the terms with $\vhat_1$ have proven useless.
In these cases we can move on to:
\begin{defn}[{\bf Long version of $d_1$}]
Following convention \ref{voeconv} when the $\vhat_1$
	term proves useless, the {\em long version} of $d_1$ is:
\label{longver}
$$
d_1(\uhat) =     \pntr^2 
\quad \mod (2, \pntr^2 \uhat) 
$$
\end{defn}

Our $d_1$  acts only on the $\uhat_k$ because $d_1$ commutes
with the $\pntr_i$ and $v_2^2$, but we show now that if we act on more
than one $\uhat_k$ at a time, the result is in a high enough length we don't need
to worry about it.

If we apply $d_1$ to two of our $\uhat$ at the same time, we get
$$
d_1(\pntr^J \uhat^r) = \vhat_1^2 \pntr^{J+\Delta_i + \Delta_j} 
\uhat^{r - \Delta_i -\Delta_j}
$$
This raises length by 2.  In our situations, if $\vhat_1$ is around, 
it would be unnecessary to use 2 different $\uhat$.
We could just use one of $\uhat_i$ or $\uhat_j$, choosing the lower
of $i$ and $j$ to get the lowest filtration element.

We need to consider the case where there is no $\vhat_1$ in the associated
graded object
on
$$
 \pntr^{J+\Delta_i + \Delta_j} \uhat^{r - \Delta_i -\Delta_j}.
$$
To get rid of a $\vhat_1$ using the formula \ref{oldformula2}, 
we have to add two more to the length, and, again, we are
out of bounds for our work, having increased the length by 4.

\begin{remark}[{\bf Only one $\uhat_k$ at a time}]
	We will never need to apply $d_1$ to more than one $\uhat_k$
at a time in each of the distinct permutations.
	This simplifies things dramatically.
\end{remark}

We need to identify the lead term of $d_1(\ww)$ in our spectral sequence for $d_1$.
We will do this inductively by computing the map $d_{1,j}$, which is just our $d_1$
in our spectral sequence, restricted to $\ww$ with $\ep_1 = \ep_2 = \ldots = \ep_{j-1} = 0$
and $\ep_j = 1$,  that is, our $W_j$ of section \ref{sketch}.

Since $d_1(\ww)$ is a symmetric function, the 
lead term must be a term of $d_1(\pntr^J \uhat^r)$,
where $\pntr^J \uhat^r$ is a distinct permutation of the 
lead term for $\ww$, i.e. $\pntr^I 
\uhat^\ep$.  
If $\pntr^J \uhat^r$ is anything other than the lead term, 
it cannot have {\bf property A}
in order to be a distinct permutation.
However, if it is going to create a lead term for $d_1(\ww)$, 
a term of 
$d_1(\pntr^J \uhat^r)$  must have {\bf property A}.

There can be many distinct permutations on our lead term to make up a $\ww$.
The two properties listed above restrict the 
permutations we need to be concerned with.

Only a few things can happen with our $d_1$.  The first thing that always
happens
is to take a $\pntr_k^{i_k} \uhat_k$ to $\vhat_1 \pntr_k^{i_k +1}$.
Sometimes this is enough because our associated graded object is
free over $\Zq[\vhat_1]$ and our choice of $k$ gives the lowest
filtration.  
Often it is not enough because the term
with $\vhat_1$ is not there in the associated graded object and
we need to apply the relation $\vhat_1 \pntr_h^{i_h} = \pntr_h^{i_h +1 }$
for some $h$ and get
\begin{equation}
\label{twostep}
\pntr^{J+\Delta_h + \Delta_k} \uhat^{r - \Delta_k}
\end{equation}
In special cases we have to go straight to the long
form of $d_1$ and take $\pntr_k^{i_k}\uhat_k$ directly to $\pntr_k^{i_k +2}$.

Unfortunately, we cannot write down a general formula that works
in all of our cases.  Our computations depend too much on the
state of the associated graded object at the time of the computation. 
There are, however, some recurring standard computations that we
can discuss.  Before we look at these general cases, it is illuminating
to look at some small special cases.  

We begin with
$
w_{(1,0),(0,1)} = \pntr_1 \uhat_2 + \uhat_1 \pntr_2,
$
which has lead term $\pntr^{(1,0)} \uhat^{(0,1)} = \pntr_1 \uhat_2$.
If we take $d_1$ of this using the short version of $d_1$,
we get
$$
d_1(\pntr_1 \uhat_2 + \uhat_1 \pntr_2) 
= \vhat_1 (\pntr_1 \pntr_2 + \pntr_1 \pntr_2)
= 2 \vhat_1 \pntr_1 \pntr_2 = 0.
$$
In cases (and there are many) like this, we call on the 
long form of $d_1$ where
we have established that we can ignore the $\vhat_1$'s.
Now we get
$$
d_1(\pntr_1 \uhat_2 + \uhat_1 \pntr_2) 
= \pntr_1 \pntr_2^2 + \pntr_1^2 \pntr_2
= w_{(2,1),(0,0)}.
$$
Our lead term for this is $\pntr_1^2 \pntr_2$, 
and this is a case where the lead
term of $d_1(\ww)$ does not come from $d_1$ on the lead term of $\ww$, 
something that would make our lives
much easier.

Stepping up to the similar situation for $n=3$, consider
$$
w_{(1,1,0),(0,0,1)}
=
\pntr_1
\pntr_2
\uhat_3
+
\pntr_1
\uhat_2
\pntr_3
+
\uhat_1
\pntr_2
\pntr_3
$$
This time, applying the short version of $d_1$ gives us
$$
\vhat_1
\pntr_1
\pntr_2
\pntr_3
+
\vhat_1
\pntr_1
\pntr_2
\pntr_3
+
\vhat_1
\pntr_1
\pntr_2
\pntr_3
= 
3\vhat_1
\pntr_1
\pntr_2
\pntr_3
$$
We have two possibilities at this point.
If the associated graded object is free over $\Zq[\vhat_1]$,
we are done.
If $\vhat_1$ is zero on the associated graded object, we
could, in principle, get $\ww$ with lead term
$\pntr_1^2 \pntr_2 \pntr_3$.  In fact, in the $n=3$ case
this doesn't happen but it still illustrates a point
because related things like this do happen when $n > 3$.
The same is true about the next example as well.

Consider
$$
w_{(2,2,0),(0,0,1)}
=
\pntr_1^2
\pntr_2^2
\uhat_3
+
\pntr_1^2
\uhat_2
\pntr_3^2
+
\uhat_1
\pntr_2^2
\pntr_3^2
$$
Start by using the short version of $d_1$ to get
$$
\vhat_1
\pntr_1^2
\pntr_2^2
\pntr_3
+
\vhat_1
\pntr_1^2
\pntr_2
\pntr_3^2
+
\vhat_1
\pntr_1
\pntr_2^2
\pntr_3^2
= 
\vhat_1 w_{(2,2,1),(0,0,0)}
$$
If this is an element, we are done.  If $\vhat_1 =0$ here, we
have to apply relation \ref{oldformula2}.  The obvious choice
gives us
$
w_{(3,2,1),(0,0,0)},
$
but if this is not an element in our associated graded object,
we would have to apply relation \ref{oldformula2} to the $i_3 = 1$
term giving us
$3w_{(2,2,2),(0,0,0)}$.

It is worth keeping these simple examples in mind as
we try to look at some general cases.

We are now going to prove some highly technical lemmas
that will help us get through our rough computations later.
Each of our $E_{1,j}$ comes in two parts, a $\Zq[\vhat_1]$
free part and a part where $\vhat_1$ is zero on the associated
graded object.  Dealing with the $\Zq[\vhat_1]$ free part
is  fairly easy, so we start with it.
We don't have to know much right now about $E_{1,j}$, except
that the elements $\ww$ all have $\ep_k = 0$ for $k < j$ and
we are only interested in computing $d_{1,j}$ on the elements
with $\ep_j = 1$.

As we will use the following lemmas in our main computation,
we abandon the use of the convention \ref{voeconv}.

\begin{lemma}[{\bf The $\vhat_1$ free part}]
\label{free}
Given $\ww \in E_{1,j}$ with $\ep_j = 1$ 
in the $\Zq[\vhat_1]$ 
free part of $E_{1,j}$ for $MU(n)$ such that either 
$$
i_{j-s} = i_{j-s+1} = \cdots = i_{j-2} = i_{j-1} = i_j + 1
$$
with $s$ maximal and even or $i_{j-1} > i_j +1$ (the equivalent
of $s=0$).  Then
$$
d_{1,j}(\voe \ww) = v_2^{-3} \voe \vhat_1 w_{I+\Delta_j,\ep -\Delta_j}
$$
\end{lemma}

\begin{proof}
First note that $(I+\Delta_j,\ep-\Delta_j)$ has \A\ because all
we changed was $i_j$ and it was raised by 1 to be less than or equal to
$i_{j-1}$.

Second, we want to show how to get such a term, and then we
will show that no other term with \A\ has a lower filtration.

We start with the $i_{j-1} = i_j + 1$ option.
We can consider all of the permutations where all we have done
is moved $\pntr_j^{i_j} \uhat_j$ to the left in the place
of $\pntr_{j-k}^{i_{j-k}}$ for $k$ from $1$ to $s$ (there is no 
$\uhat_{j-k}$ by \A\ and the description of $E_{1,j}$).  When we apply
our short $d_1$ to each of these terms, with our $\uhat_j$
in the $j-k$ place, we have $(s+1)$ terms all the same
as our desired result.  Since $s$ is even, we have our
required term.

If $i_{j-1} > i_j + 1$, we can just apply the short $d_1$
to $\uhat_j$ to get the required term.
Note here that if we try to shift the $\uhat_j$ term
to the left, we get a term without \A, such that
when we apply the short $d_1$ to it, it still does not have
\A.
This is really just the $s=0$ version of the lemma.

Now we have to show that we cannot achieve a lower filtration
element using any other $\uhat_k$ and/or permutation.

We pick a 
$\pntr_k^{j_k} \uhat_k$ 
in some permutation,
$\pntr^J \uhat^r$ of $\pntr^I \uhat^\ep$
to apply our short $d_1$ to.  If we remove 
$\pntr_k^{j_k} \uhat_k$ 
from $\pntr^J \uhat^r$, we must have \A.
If not, we cannot get \A\ when we apply $d_1$ to $\uhat_k$.
And so, what remains, must be a subsequence of $\pntr^I \uhat^\ep$
with just one term missing, 
$\pntr_h^{i_h} \uhat_h$.  
The
permutation is to just move 
$\pntr_h^{i_h} \uhat_h$
to
$\pntr_k^{j_k} \uhat_k$
leaving all other terms fixed.
By this we mean that $i_h = j_k$.
If $h < k$, we have moved
$\pntr_h^{i_h} \uhat_h$
to the right.  For this to be a distinct
permutation, we must have $2i_h + 1 > 2i_k + \ep_k$.
It is because of this term that this distinct permutation 
has a higher filtration than the lead term.
Since we are going to then replace $\uhat_k$ with $\vhat_1 \pntr_k$,
we are going to increase the filtration even further.
Since this situation can only happen when $j \le h < k$, 
($\ep_h =0$, $h < j$),
this is
of a higher filtration than the element we have already
discussed.

We have shown that, in this case, the only relevant 
permutation consist of sliding some 
$\pntr_k^{j_k} \uhat_k$
to the left, because we have shown that going to the
right results in higher filtration elements.

Our first computation involves $\uhat_j$ and permutations
that involve sliding it to the left, so all we have to
do now is eliminate sliding $\uhat_k$ to the left when
$j < k$.  To get a distinct permutation, we must
have $2i_{k-1} +\ep_{k-1} > 2i_k + 1 (= \ep_k)$.
We must
slide the term in the $k$-th place passed the one
in the $(k-1)$-st place and then
apply the $d_1$ to
the moved $\uhat_k$.
That gives us the same length, but the increase in the
$k$-th place by this permutation gives it a higher filtration
than the term we have already obtained.
\end{proof}

\begin{remark}[{\bf Limits on permutations}]
\label{limits}
The above lemma took care of all of the $\Zq[\vhat_1]$
issues we will come up against.  The differential $d_{1,j}$ on the part of 
$E_{1,j}$ with $\ep_j = 1$ and $\vhat_1$ equal to zero on it always raises
the length of $(I,\ep)$ by 3 either because we use the long
version of $d_1$ or the short version followed by the
relation \ref{oldformula2}.  To compare filtrations, we have
to use the criteria for the order on the $(I,\ep)$ other
than the length.  

We want to limit the types of permutations we need to consider.
We only look at the two step process where we use the short
$d_1$ and then the relation.  The proof of the case using the long $d_1$ is
similar to the previous lemma.

The first assumption we make is that we can 
find a non-zero $w_{I+\Delta_h+\Delta_j,\ep - \Delta_j}$ term 
in $d_1(\ww)$ with $h < j$ with \A.  We will have to do this with
computations in our lemmas, but we just assume it here.

Now we want to eliminate all but a few permutations from
our consideration.                     

Consider some permutation, $\pntr^J \uhat^r$, of our lead term,
$\pntr^I \uhat^\ep$.  We plan on applying the short $d_1$
to some $\uhat_k$ and then using relation \ref{oldformula2}
on some $\pntr_h^{j_h}$.  If we remove these two terms from
$\pntr^J \uhat^r$, what remains must have \A, and so is a
subsequence of $\pntr^I \uhat^\ep$.  
Consequently, we can describe our permutation of $\pntr^I \uhat^\ep$
to $\pntr^J \uhat^r$ as just moving two terms around, namely
some $\pntr_{k'}^{i_{k'}} \uhat_{k'}$ moving to 
$\pntr_{k}^{j_{k}} \uhat_{k}$ with $i_{k'} = j_k$ and some
$\pntr_{h'}^{i_{h'}} \uhat_{h'}^{\ep_{h'}}$ moving to 
$\pntr_{h}^{j_{h}} \uhat_{h}^{r_h}$ with $i_{h'} = j_h$ and 
$\ep_{h'} = r_h$.  All our permutation does is slide these
two terms around, either to the left or right in $\pntr^I \uhat^\ep$.

Because we have assumed the existence of a certain type of element
in $d_1(\ww)$, we can see immediately that any 
change to the right of the $\uhat_j$ place, either due to $d_1$
or the permutation, will result in a higher filtration term, much
as in the previous lemma.

Since we can't mess with things to the right of $\uhat_j$, 
we must have $k' = j$.
The only
permutation that $\pntr_j^{i_j} \uhat_j$ can be involved with is
a shift to the left.  Likewise, the 
$\pntr_{h'}^{i_{h'}} \uhat_{h'}^{\ep_{h'}}$ term above cannot
be to the right of the $\uhat_j$ term, but must be to the left.
That means that $\ep_{h'} = r_h = 0$.  If we try to shift
our 
$\pntr_{h'}^{i_{h'}}$
to the right, we automatically end up with something of higher
filtration again, so this term too must shift only to the left
if at all. 

There are limitations when shifting to the left as well.  
If we try to shift 
$\pntr_j^{i_j} \uhat_j$ 
to the left,
we can only go passed terms with $i_k = i_j+1$.  Otherwise,
when we change 
$\pntr_j^{i_j} \uhat_j$ 
to
$\pntr_j^{i_j+1} $ 
we would not have \A.
Similarly, if we try to shift $\pntr_h^{i_h}$ to the left, it
can only go passed terms with $i_{h'} = i_h +1$ or we will
not have \A\ when we apply relation \ref{oldformula2}.
\end{remark}

\begin{lemma}
\label{nonfree1}
Given $\ww \in E_{1,j}$ with $\ep_j = 1$ 
in the 
part of $E_{1,j}$ of $MU(n)$ that has $\vhat_1 =0$ on it
such that 
$$
i_{j-s} = i_{j-s+1} = \cdots = i_{j-2} = i_{j-1} = i_j + 1
$$
with $s$ maximal and odd and  
$$
i_{j-s-t} = i_{j-s-t+1} = \cdots = i_{j-s-2} = i_{j-s-1} = i_{j-s} + 1
$$
with $t$ maximal and even.
Then
$$
d_{1,j}(\voe \ww) = v_2^{-3} \voe  w_{I+\Delta_{j-s} +\Delta_j,\ep -\Delta_j}
$$
\end{lemma}

\begin{proof}
First note that $(I+\Delta_{j-s} +\Delta_j,\ep-\Delta_j)$ 
has \A. 
The $\Delta_j$ part is for the same reason as in the
previous lemma.  We also know that $i_{j-s-1} > i_{j-s}$
by definition, so adding the $\Delta_{j-s}$ preserves \A.

Second, we want to show how to get such a term, and then we
will show that no other term of $d_{1,j}(\ww)$ with \A\ has 
a lower filtration.

We can consider all of the permutations where all we have done
is moved $\pntr_j^{i_j} \uhat_j$ to the left in the place
of $\pntr_{j-k}^{i_{j-k}}$ for $k$ from $1$ to $s$ (there is no 
$\uhat_{j-k}$).  When we apply
our short $d_1$ to each of these terms, with our $\uhat_j$
in the $j-k$ place, we have $(s+1)$ terms all the same,
but this time, we have an even number of them and so this is
zero.
So, the $\vhat_1$ part of $d_1$ has proven useless on these
terms.  Moving on to the long form of $d_1$, we replace
the $\uhat_j$ with $\pntr_{j-k}$ in each $(j-k)$ place
of the various permutations.  
These terms are all now in different filtrations.
The lowest filtration version
gives the answer we are looking for.

The above covers the $t=0$ case, i.e. where 
$i_{j-s-1} > i_{j-s} + 1$
and deals with the first few possible permutations
of the $ t > 0$ case, i.e. where
$i_{j-s-1} = i_{j-s} + 1$.
In this case though, there are other possible permutations.
We cannot do anything with $i_b$ where $b < j-s-t$ because
we have already used the long $d_1$ and there is nothing
else to do.  However, we can shift $i_{j-s}$ to the left
from 1 to $t$ times.  Then our permutation on the $(I,\ep)$
of $\pntr^I \uhat^\ep$ looks like
$$
(I-\Delta_{j-s -c} + \Delta_{j-s},\ep)
$$
for each $c$ from 1 to $t$.  For each such $c$, we can consider
the permutations that just slides $\pntr_j^{i_j} \uhat_j$ to
the left, but we can now only do this $(s-1)$ times, giving
us a total of $s$ equal terms.  Since $s$ is odd, this gives
us 
$$
v_2^{-3} \voe \vhat_1 \pntr^{I-\Delta_{j-s -c} + \Delta_{j-s} + \Delta_j} 
\uhat^{\ep - \Delta_j}
$$
To make this have \A, we have to apply relation \ref{oldformula2}
to $\vhat_1 \pntr_{j-s-c}$.  Together with the first case
that left $\pntr_{j-s}^{i_{j-s}}$ where it was, we have $(t+1)$
of these, but since $t$ is even, our final result is the
desired
$$
v_2^{-3} \voe  \pntr^{I+ \Delta_{j-s} + \Delta_j} 
\uhat^{\ep - \Delta_j}.
$$

Now we have to show that we cannot achieve a lower filtration
element 
in this situation
using any other $\uhat_k$ and/or permutation.

Remark \ref{limits} restricted the permutations we needed to
deal with.  It forced us to start with $\uhat_j$  for $d_1$ and then
deal with $\pntr_h$ with $h < j$ with the relation \ref{oldformula2}
if need be.
This is indeed, exactly what we did, so we see that this is
the only possibility.
\end{proof}

\begin{lemma}
\label{nonfree2}
We start with $\ww \in E_{1,j}$ with $\ep_j = 1$ 
in the 
part of $E_{1,j}$ for $MU(n)$ that has $\vhat_1 =0$ on it.
We assume that
$$
i_{j-s} = i_{j-s+1} = \cdots = i_{j-2} = i_{j-1} = i_j + 1
$$
with $s$ maximal and even.  
We also assume that, for some $k < j-s$, we have
$$
i_{k-t} = i_{k-t+1} = \cdots = i_{k-2} = i_{k-1} = i_k + 1
$$
with $t$ maximal and even.  
We further assume that $k$ is the smallest number such
that $w_{I+\Delta_{k} + \Delta_j,\ep - \Delta_j}$ is
in $E_{1,j}$.
Then
$$
d_{1,j}(\voe \ww) = v_2^{-3} \voe  w_{I+\Delta_{k} +\Delta_j,\ep -\Delta_j}
$$
\end{lemma}

\begin{remark}
\label{ts}
This seems highly technical, but it covers a lot of
territory for us.
It even covers more than is obvious.  If $s = 0$, that
is the same a $i_{j-1} > i_j + 1$ and if $t=0$, that
is the same as $i_{k-1} > i_k + 1$.
\end{remark}

\begin{proof}
It is easy to see that our term has \A.  We just need to
see that we can obtain it, but by now, that is straight
forward.  With $s$ even, we know the permutations from
lemma \ref{free} that give us the short $d_1$ on our
lead term along with these permutations.  
Note that as in remark \ref{ts}, this is even easier if $s=0$
as there are no relevant permutations.
We get
$$
v_2^{-3} \voe \vhat_1 w_{I+\Delta_j,\ep -\Delta_j}
$$
Now, using similar permutations and $t$ even, we can
apply relation \ref{oldformula2} to the $(t+1)$ permutations
to get the same term, namely the desired
$$
v_2^{-3} \voe  w_{I+\Delta_{k} +\Delta_j,\ep -\Delta_j}.
$$

We have to rule out one possible glitch.  If $i_{j-s}+1 = i_{j-s-1}$,
we could try to shift the term in the $(j-s)$ place to the
$(j-s-1)$ place or lower, we could have something like what happened
in the previous lemma, but we don't.  If we do this, the
possible shifts on the term in the $j$-th coordinate are to
move it to the left from 1 to $(s-1)$ times.  This would give
$s$ identical terms when we applied the short $d_1$, but $s$
is even, so we would have to go to the long $d_1$.  Using the
same argument as the previous lemma, that would raise 
$i_{j-s+1}$ by one, and this would make it automatically have
a higher filtration than the term we have already found.
\end{proof}

\section{Computing $d_{1,j}$, low $j$, for $MU(n)$}

We recall the definition of \A.
$$
2i_1 + \ep_1 \ge \cdots \ge
2i_k + \ep_k \ge 2i_{k+1} + \ep_{k+1} \ge \ldots \ge 2i_n + \ep_n > 0
$$

We start the computation of $d_1$ on $E_{1,1}$  
only using the $\ww$ with $\ep_1 = 1$.
We call this map $d_{1,1}$ and
the result of this computation, $E_{1,2}$. 
This is all very similar to the work in section \ref{d11} but we have to 
contend with the symmetric function now in our computation. 

\begin{prop}
	\label{d11a}
	With  $\ep_1 = 0$ and {\bf property A},
	$E_{1,2}  $ for $MU(n)$ is:
$$
\Zq[\vhat_1 ]\{\voe v_2^{0,2,4,6} w_{I,\ep} \}
\quad i_1 = i_2
$$
and
$$
\Zq\{\voe v_2^{0,2,4,6} \ww \}
\quad i_1 > i_2
$$
The $x^1$-torsion generators detected by $d_{1,1}$ are represented by:
$$
\Zp[\vhat_1]\{\vhat_1 \voe  v_2^{0,2,4,6} \ww \}
\quad i_1 > i_2
$$
\end{prop}

\begin{proof}
Recall that we are now working mod (2) and that $d_1$ 
commutes with $\pntr_i$ and $v_2^2$, so we
can concentrate on $\voe \ww$ from $E_{1,1}$ with $\ep_1 = 1$.

All we have to do is apply lemma \ref{free} with $s=0$,
giving us:
$$
d_{1,1}(\voe \ww) = v_2^{-3} \voe \vhat_1 w_{I+\Delta_1,\ep - \Delta_1}.
$$

Note that the first part of $E_{1,2}$ is there because $i_1 = i_2$ with $\ep_1 = 0$ (and
therefore $\ep_2 = 0$), cannot be the target of our differential.
The result follows.
\end{proof}

\begin{remark}
If $n=1$,
the above
is consistent with the results for $ER(2)^*(\cp)$ from
\cite[Theorems 3.1 and 4.1]{Kitch-Lor-Wil-CPn}, i.e. the $n=1$ case,
even if, at first glance, they don't look the same.
Here, the only $\ww$ we have left for $E_2$ are the $\pntr_1^i$, which is the
same as $\Pntr_1^i$.
\end{remark}



Our proofs can generously be called tedious.  More detail would
not make them more user friendly.  The die-hard reader who really
cares about the details will have to put in serious effort.
To begin the induction, it isn't necessary to compute all
of the $E_{1,3-5}$, but, speaking from experience, they are
invaluable guides to the general inductive case and so
we have left them in.

\begin{prop}
	\label{d12a}
	 With  $\ep_1 = \ep_2 = 0$ and {\bf property A},
	 $E_{1,3} $ for $MU(n)$ is:
$$
\Zq[\vhat_1 ]\{\voe v_2^{0,2,4,6} w_{I,\ep} \}
\quad i_1 = i_2
$$
and
$$
\Zq\{\voe v_2^{0,2,4,6} \ww \}
\quad i_1 > i_2 = i_3
$$
The $x^1$-torsion generators detected by $d_{1,2}$ are represented by:
$$
\Zp\{\voe  v_2^{0,2,4,6} \ww \}
\quad i_1 > i_2 > i_3.
$$
\end{prop}

\begin{proof}
Because $\ep_2 = 0$ already on the first part of $E_{1,2}$, we have no $d_{1,2}$
on this part.

For the second part, with $i_1 > i_2$, our proof comes in two stages.  
First we assume that $i_1 > i_2 + 1$.
In this case we just apply lemma \ref{nonfree2} 
with $s=t=0$ and $k=1$
to get
$$
v_2^{-3} \voe w_{I+\Delta_1 + \Delta_2,\ep - \Delta_2}.
$$
If $i_1 = i_2 +1$, we use lemma \ref{nonfree1} 
with $s = 1$
to get
the same result.
This eliminates the $i_1 > i_2$ terms with $\ep_2 = 1$ as
sources and the $i_1 > i_2$ terms with $\ep_2 = 0$
as targets, missing only the $i_2 = i_3$ terms.  This concludes
the proof.
\end{proof}

\begin{remark}
If $n=2$, we would be done computing an associated graded version of $E_2$ for
the Bockstein spectral sequence.  There appear to be two parts to the answer,
but there are no $\ww$ with $i_2 = i_3$ because there is no $i_3$.  Consequently,
the answer is entirely in the first part, namely
$$
\Zq[\vhat_1]\{v_2^{0,2,4,6} \ww\} \quad i_1 = i_2 \quad \ep_1 = \ep_2 = 0.
$$
These $\ww$ are no more than just $\pntr_1^{i} \pntr_2^i \in E(2)^*(\Smash^2 \cp)$,
which is the image of $\Pntr_2^i \in E(2)^*(MU(2))$.
\end{remark}

\begin{prop}
	\label{d13a}
	 With  $\ep_1 = \ep_2 = \ep_3 = 0$ and {\bf property A},
	 $E_{1,4} $ for $MU(n)$ is:
$$
\Zq[\vhat_1 ]\{\voe v_2^{0,2,4,6} w_{I,\ep} \}
\quad i_1 = i_2
\quad i_3 = i_4
$$
and
$$
\Zq\{\voe v_2^{0,2,4,6} \ww \}
\quad i_1 > i_2 = i_3
$$
$$
\Zq\{\voe v_2^{0,2,4,6} \ww \}
\quad i_1 = i_2  
\quad i_3 > i_4  
$$
The $x^1$-torsion generators detected by $d_{1,3}$ are represented by:
$$
\Zp[\vhat_1]\{\vhat_1 \voe  v_2^{0,2,4,6} \ww \}
\quad i_1 = i_2 
\quad i_3 > i_4  
$$
\end{prop}

\begin{proof}
This one is fairly easy.  For the second part of $E_{1,3}$ we have $i_2 = i_3$,
but we also have $\ep_2 = 0$, so we must also have $\ep_3 = 0$.  Therefore there is
no $d_{1,3}$ on this second part.

As for the first part, because we want to consider $\ep_3 = 1$ with $\ep_1 = \ep_2 = 0$,
we must have $2i_2 \ge 2i_3 +1 (= \ep_3)$, so 
$i_2 > i_3$.  
Applying $d_{1,3}$ using lemma \ref{free}, we get
$$
v_2^{-3} \voe
\vhat_1 w_{I+\Delta_3,\ep - \Delta_3}
$$
This leaves our conditions $i_1 = i_2$ and $i_3 = i_4$ on the first part (because
they are missed), and
the quotient of $d_{1,3}$ on the first part gives us the $i_1 = i_2$, $i_3 > i_4$
of the second part.
\end{proof}

\begin{remark}
If $n=3$, we are done.  Because in the first part, $i_3 = i_4$ 
and there is no $i_4$, there is no 
contribution to the answer from this first part.

For the second part, we can always write our answer in terms of:
$$
\chat_1^{2(i_1 - i_2)} \chat_2^{2(i_2-i_3)} \chat_3^{2i_3}
\quad i_3 > 0.
$$
We have conditions on $i_j$.  
In the first case with $i_1 > i_2 = i_3$, we get
$$
\Pntr_1^i \Pntr_3^j \quad i, j > 0
$$
In the second case we have $i_1 = i_2 \ge i_3$.
This gives us
$$
\Pntr_2^{i} \Pntr_3^{j}  
\quad  i \ge  0
\quad  j >  0
$$
\end{remark}

This last example can be used to ground our induction.

\begin{prop}
	\label{d14a}
	 With  $\ep_1 = \ep_2 = \ep_3 = \ep_4 = 0$ and {\bf property A},
	 $E_{1,5} $ for $MU(n)$ is:
$$
\Zq[\vhat_1 ]\{\voe v_2^{0,2,4,6} w_{I,\ep} \}
\quad i_1 = i_2
\quad i_3 = i_4
$$
and
$$
\Zq\{\voe v_2^{0,2,4,6} \ww \}
\quad i_1 > i_2 = i_3
\quad i_4 = i_5
$$
$$
\Zq\{\voe v_2^{0,2,4,6} \ww \}
\quad i_1 = i_2 
\quad i_3 > i_4 = i_5 
$$
The $x^1$-torsion generators detected by $d_{1,4}$ are represented by:
$$
\Zq\{\voe v_2^{0,2,4,6} \ww \}
\quad i_1 > i_2 = i_3
\quad i_4 > i_5
$$
$$
\Zq\{\voe v_2^{0,2,4,6} \ww \}
\quad i_1 = i_2 
\quad i_3 > i_4 > i_5 
$$
\end{prop}

\begin{proof}
The easy part is the first part, we must have $\ep_4 = 0$, so
there is no differential.
On the rest, there are many cases to consider.  Note that after we apply
$d_1$ to $\uhat_4$, we can never hit $i_4 = i_5$ (because of \A), so we will have that
condition in the end.

We first look at the $i_1 > i_2 = i_3$ part of $E_{1,4}$.
By \A, we also have $i_3 > i_4$.
If $i_4 + 1 = i_3$ we use lemma \ref{nonfree2} with
$s=2$, $t=0$, and $k=1$, to get
$$
v_2^{-3} \voe
w_{I+\Delta_1 + \Delta_4,\ep-\Delta_4}.
$$
If $i_4 + 1 < i_3$, we use lemma \ref{nonfree2} with $s=t=0$ and
$k=1$ to get the same result.  It wasn't really necessary to
break this into two pieces since lemma \ref{nonfree2} handled
both.

This gives us everything 
in the first part of our non-$\vhat_1$ part of $E_{1,5}$
except when $i_1 = i_2+1$.  
We already had $i_1 > i_2 $ and we added $1$ to $i_1$.
We can fix this by looking at the second part when we have $i_1 = i_2 = i_3$.
We know $i_4 < i_3$.  If $i_4 + 1 = i_3$, we use lemma \ref{nonfree1}
with $s=3$ and $t=0$.  If $i_4 + 1 < i_3$, we use lemma \ref{nonfree2} with
$s=t=0$ and $k=1$.
This now gives us our $i_1 = i_2 + 1$ case.

It is time to take stock of where we are. 
We have acquired
all of the first part of our answer and used 
up the $i_1 = i_2 = i_3 > i_4$
part of the second part of $E_{1,4}$ as sources. 

We still need to hit, as targets, 
all of the $\ww$ with $i_1 = i_2 \ge i_3 > i_4 > i_5$
when $\ep_4 = 0$.  
The $i_4 > i_5$ always takes care of itself.

For sources, we need to use the $i_1 = i_2 > i_3 > i_4$ with 
$\ep_4 = 1$.
It will complete the proof if we can show that for these source
$(I,\ep)$, we have:
$$
d_{1,4}(\voe \ww) = v_2^{-3} \voe w_{I+\Delta_3 + \Delta_4,\ep - \Delta_4}.
$$
We cannot replace the $\Delta_3$ with $\Delta_1$ because our
element would not be in $E_{1,4}$.  If we try to replace it with
$\Delta_2$, the term does not have \A.
If $i_3 > i_4 + 1 $, we just apply lemma \ref{nonfree2}
with $k=3$, $s=0$ and $t=0$ unless $i_2 = i_3 + 1$, in which case
we use $t=2$. 
If
$i_3 = i_4 + 1$, we use lemma \ref{nonfree1} with $s=1$
and $t=0$ unless $i_2 = i_3+1$, in which case we use $t=2$.
\end{proof}

\begin{remark}
If $n=4$, we are done.  Because in the second part, $i_4 = i_5$ 
and there is no $i_5$, there is no 
contribution to the answer from this second part.

For the first part, 
our lead term for $\ww$ is just
$\pntr_1^i \pntr_2^i \pntr_3^j \pntr_4^j$ with $i \ge j > 0$.
This is the image of $\Pntr_2^{(i-j)} \Pntr_4^j$.
\end{remark}

\section{Computing $E_{1,j+1}$ for $MU(n)$}
\label{dobu}

We recall the definition of \A.
$$
2i_1 + \ep_1 \ge \cdots \ge
2i_k + \ep_k \ge 2i_{k+1} + \ep_{k+1} \ge \ldots \ge 2i_n + \ep_n > 0
$$

We are using an auxiliary spectral sequence that comes from the
filtration defined by the ordering on the $(I,\ep)$ 
to compute the $d_1$ for the
Bockstein spectral sequence.
Following our description of the process in section \ref{sketch},
we compute our spectral sequence for $d_1$ by induction on $j$ using the
$\ww$
with $\ep_k = 0$ for $k < j$ and $\ep_j = 1$, i.e., the $W_j$ of
section \ref{sketch}.  
We call this
map $d_{1,j}$ and it is defined on $E_{1,j}$ and the result gives us
$E_{1,j+1}$.
As in section \ref{sketch}, the map $d_{1,j}$ is injective on $W_j$ so we are left
with $\ep_j = 0$ in $E_{1,j+1}$.
When we have done $d_{1,n}$ and computed
$E_{1,n+1}$ (as a quotient of $W_{n+1}$), 
we will be done, giving an associated graded version of the $E_2$ of
the Bockstein spectral sequence.
Since at this stage all $\ep_k = 0$, $s(\ep)=0$ and is even, making $\voe = 1$.

\begin{thm}
\label{bigthm}
	For the spectral sequence
	for the calculation of $E_2$ for
	the Bockstein spectral sequence from $E(2)^*(MU(n))$ to $ER(2)^*(MU(n))$, 
	we always have {\bf property  A}.
For	$E_{1,j+1}$, $1 \le j \le n$, we have
$
\ep_1 = \ep_2 = \cdots = \ep_{j} = 0.
$
There are two parts to $E_{1,j+1}$.  First:
$$
\Zq[\vhat_1]\{\voe v_2^{0,2,4,6} w_{I,\ep}\}
\quad \text{with} \quad
i_{2b-1} = i_{2b}  \quad 0 < 2b \le j+1
$$
Second,
for $b$ with $0 < 2b+2 \le j+1$, let :
$$
i_{2c-1} = i_{2c} \quad 0 < 2c \le 2b,
\quad
i_{2b+1} > i_{2b+2},
\quad
i_{2a} = i_{2a+1} \quad 2b < 2a < j+1
$$
Then we have:
$$
\Zq\{\voe v_2^{0,2,4,6} w_{I,\ep}\}
$$
When $j=2q+1$, the $x^1$-torsion detected by $d_{1,j}$ is represented by:
$$
\Zq[\vhat_1]\{\vhat_1 \voe  v_2^{0,2,4,6} \ww \}
\quad 
i_{2b-1} = i_{2b}  \quad 0 < b \le q
\quad 
i_j > i_{j+1}
$$
When $j=2q$, the $x^1$-torsion detected by $d_{1,j}$ is the same as
the second part of $E_{1,j+1}$ but with $i_j > i_{j+1}$.
\end{thm}

\begin{remark}
It is easy enough to read off the terms in the theorem that
are in degrees $8*$.  It requires $s(\ep)$ to be even, forcing
$\voe = 1$.  Then just eliminate the $v_2^{2,6}$ as well.
To get just terms in degrees $16*$, also eliminate $v_2^4$.
All $x^i$-torsion generators inject to $E(2)^*(-)$, so we
see that the $x^1$-torsion generators of degree $8*$ inject,
giving part of theorem \ref{inj}.
\end{remark}

\begin{remark}
When $j= n = 2q+1$, the condition on the 
$\Zq[\vhat_1]$ free
part has $i_n = i_{n+1}$, but since there is no $i_{n+1}$, this
condition is never met and there is no 
$\Zq[\vhat_1]$ free part.
When $j = n = 2q$, the condition on the 
part with $\vhat_1 = 0$
has $i_n = i_{n+1}$, but since there is no $i_{n+1}$, this
condition is never met and there is no 
part with $\vhat_1 = 0$.
\end{remark}

\begin{proof}
Our proof is by induction.  We assume we have computed
$d_{1,j'}$ for $j' < j$.  We need to compute $d_{1,j}$
on $E_{1,j}$ and show our result gives $E_{1,j+1}$.
We have computed $E_{1,2}$ through $E_{1,5}$ to begin our induction.
In fact, we need the $d_{1,4}$ to ground our induction.

There are some, but not enough, easy parts to this.  First, if
$j=2q$, $d_{1,j} = 0$ on the first part because we have
$i_{j-1} = i_j$ and so $\ep_j = 0$.  Likewise, if
$j=2q+1$, 
$d_{1,j} = 0$ on the second part because we have
$i_{j-1} = i_j$ and so $\ep_j = 0$.  

When 
$j=2q+1$, computing $d_{1,j}$ on the first part is just
lemma \ref{free}.
This misses the usual $i_j = i_{j+1}$, but, because the $\vhat_1$ is
there, this creates the $b=q$ part of $E_{1,j+1}$ in the second part, the only piece of
the second part that wasn't there already in $E_{1,j}$.  
The rest of $E_{1,j}$ remains unchanged and carries over to $E_{1,j+1}$.

What remains now is to deal with $j=2q$.  The $\Zq[\vhat_1]$ free
 part of $E_{1,j}$ is
uninvolved and carries over to be exactly the same for the first part of
$E_{1,j+1}$.

In the second part of $E_{1,j}$, the range of $b$ does not change between
$E_{1,j}$ and $E_{1,j+1}$.  However, the change does allow for $a$ to
be $q$, giving $i_j = i_{j+1}$.  
We expect this and can now forget about it.
To compute $d_1$ on $\uhat_j$, we can never end up with $i_j = i_{j+1}$, which
explains how this condition comes about.
Otherwise, the descriptions of $E_{1,j}$ and $E_{1,j+1}$ are the same 
except, of course,
we end up with $\ep_j = 0$.

Let's take a look at what we have to accomplish yet.
We have to compute $d_{1,j}$ in such a way that all the $\uhat_j$
go away.
Our map $d_{1,j}$ has to take the second part of $E_{1,j}$ with
$\ep_j = 1$ and $i_{2q} \ge i_{2q+1}$ (sources) and 
put it in 1-1 correspondence with the second part of
$E_{1,j}$, with $i_{2q} > i_{2q+1}$ 
(targets) and $\ep_j = 0$.  Recall that our $2q = j$.

First let us work with the
$b = 0$ case. 
We want all $b=0$ terms with $\ep_j = 0$ and $i_j > i_{j+1}$ to
be hit as targets.  We need to find the sources to do this with.
Our sources must have $\ep_j = \ep_{2q} = 1$, so we
have $i_{2q-1} > i_{2q}$ by \A\  and the fact that $\ep_{2q-1}=0$.
We first restrict our attention to source terms with $b=0$.

We use lemma \ref{nonfree2} to get
$$
d_{1,j}(\ww) = w_{I+\Delta_1 + \Delta_j , \ep - \Delta_j}.
$$
In this application, the $t$ of lemma \ref{nonfree2} is zero
and $k=1$,
but the $s$ could range from $0$ to $j-2 = 2q-2$ (by twos)
depending
on $I$.  This hits all elements 
in $E_{1,j}$
we need to have as targets
with $b=0$ 
and $i_1 > i_2 + 1$. 

As targets, we have not yet hit the $b=0$ terms with
$i_1 = i_2 +1 $, i.e. 
$(I,\ep)$ with $\ep_j = 0$,
$i_j > i_{j+1}$ and 
$i_1 = i_2 +1 $.
The source that works 
here is $(J,r) = (I-\Delta_1 -\Delta_j,\ep+\Delta_j)$.
To see this,
recall that for $b=0$, we have $i_{2a} = i_{2a+1}$ for $0 < 2a < 2q$.
Find the $q > b'> 0$ such that 
$$
i_1 - 1 = i_2 = \cdots = i_{2b'+1} > i_{2b'+2}
$$
In almost all cases, we can apply lemma \ref{nonfree2} to $(J,r)$
to get the desired result  
using $k=1$, 
$t=0$, and $s$ can go from $0$ to
$2q -2b' -2$ by twos, depending on $I$.

There is
one place where lemma \ref{nonfree2} does not apply and
we must use lemma \ref{nonfree1}.  That is when $b' = q-1$
and
$i_{2q-1} = i_{2q} +1$. 
Here $s=2q-1$ and $t=0$.

Note that this turns a term associated with $b' > 0$ into
one with $b=0$.

For targets, we have hit all of our $b=0$, $\ep_j = 0$, $i_j > i_{j+1}$.
For sources, we have used all of $b$ with
$i_1 = \cdots = i_{2b+1}> i_{2b+2}$ and $\ep_j = 1$, $i_j \ge i_{j+1}$
for $b = 0$ to $q-1$.
Note that this includes all of the $b=0$, $\ep_j = 1$, $i_j \ge i_{j+1}$
terms.

\begin{summary}
\label{b1}
The unused terms we need as sources are
all of the $q > b \ge 1$, $\ep_j = 1$, with
$i_j \ge i_{j+1}$, excluding terms with
$$
i_1 = i_2 = \cdots = i_{2b} = i_{2b+1} > i_{2b+2}
$$
The unused terms we need as targets are $b \ge 1$,
$\ep_j = 0$, with $i_j > i_{j+1}$.
\end{summary}

We must now do $b > 0$.

Moving on, we want to find all of the $b=1$ terms as targets.
We do much
that is similar to the $b=0$ case.
We begin with source terms that also have $b=1$.
When $b=1$, we have $i_3 > i_4$, and since we have
excluded $i_1 = i_2 = i_3 > i_4$, we always have
$i_1 = i_2 > i_3 > i_4$.
Clarity is often thwarted by the necessity to handle
special cases.
We want to apply our lemmas to get
$$
d_{1,j}(\ww) = w_{I+\Delta_3 + \Delta_j , \ep - \Delta_j}.
$$
We see that this has \A\ because $i_2 > i_3$ and $i_{j-1}>i_j$.
We cannot replace $\Delta_3$ with $\Delta_1$ because that
term does not exist in $E_{1,j}$.  We cannot replace it
with $\Delta_2$ because that term does not have \A.

Generally, we can do this using lemma \ref{nonfree2} when
we are not dealing with the special cases.  In our use
we have $t=0$ or $t=2$ (if $i_2 = i_3 +1$), 
$k=3$, and $s$ can be anywhere from $0$
to $2q-4$ (by twos).  

In the special case of source with
$j=4$ and $i_1 = i_2 > i_3 = i_4+1$ and $\ep_4 = 1$,
we have to use lemma \ref{nonfree1} with $s = 1$, $k=3$,
and $t = 0$ unless $i_2 = i_3 +1$, in which case $t = 2$.

We had $i_3 > i_4$ and we added $\Delta_3$
so we missed the cases where $i_3 = i_4 +1$.
We are left with the need to hit these cases. 
Again, this is just like the $b=0$ case.
As targets, we have not yet hit the $b=1$ terms 
$(I,\ep)$ with $\ep_j = 0$,
$i_j > i_{j+1}$ and 
$i_3 = i_4 +1 $.
The source that works 
here is $(J,r) = (I-\Delta_3 -\Delta_j,\ep+\Delta_j)$.
To see this,
recall that for $b=1$, we have $i_{2a} = i_{2a+1}$ for $2 < 2a < 2q$.
Find the $q > b'> 0$ such that 
$$
i_3 - 1 = i_4 = \cdots = i_{2b'+1} > i_{2b'+2}
$$
In almost all cases, we can apply lemma \ref{nonfree2}
to get the desired result.  
using $k=3$, 
$t=0$ or $t=2$ (if $i_2 = i_3 +1$), and $s$ can go from $0$ to
$2q -2b' -2$ by twos, depending on $I$.

Of course, if $2b'  + 1 = 2q -1$ AND $i_{2q-1} = i_{2q} + 1$,
then we have to use lemma \ref{nonfree1}. 
Here we have $s = 2q -3$, $t = 0$ or $t=2$ (if $i_2 = i_3 + 1$).

We need to identify all of the targets hit so far and
all of the sources used so far.

We have hit all elements as targets with $b=0$ or $b=1$, $\ep_j = 0$
and $i_j > i_{j+1}$.

We have used all terms as sources with $b = 0$ and $b=1$ with $\ep_j = 1$
and $i_j \ge i_{j+1}$. In addition, we have used
all terms with 
$i_1 = \cdots = i_{2b'+1} > i_{2b'+2}$ for $b' > 0$
and all terms with 
$i_1 = i_2 > i_3 = \cdots = i_{2b'+1} > i_{2b'+2}$ for $b' > 1$.
Combined, that is $i_1 =i_2 \ge i_3 = \cdots = i_{2b'+1}> i_{2b'+1}$.

\begin{summary}
\label{b2}
The unused terms we need as sources are
all of the $q > b \ge 2$, $\ep_j = 1$,
$i_j \ge i_{j+1}$, excluding terms with
$$
i_1 = i_2 \ge i_3 = i_4 =   \cdots = i_{2b} = i_{2b+1} > i_{2b+2}
$$
The unused terms we need as targets are $b \ge 2$,
$\ep_j = 0$, with $i_j > i_{j+1}$.
\end{summary}

We are getting close to our induction statement
where we will set things up to do
 $d_{1,j}$ for $b \ge 2$ using the induction.

Our $d_{1,j}$ on what is left cannot involve $i_1$ or $i_2$ because
$(I+\Delta_1 + \Delta_j,\ep-\Delta_j)$ does not give
a term in $E_{1,j}$ and 
$(I+\Delta_2 + \Delta_j,\ep-\Delta_j)$ does not
have \A.

Thus, we can ignore $i_1$ and $i_2$.  What is left of $(I,\ep)$
if we remove them is an $I'$ of length $n -2$.  More
importantly, $i_j = i_{2q}$ moves down to the
new $i'_{2q-2}$ and the $b \ge 2$ condition moves
to a $b' \ge 1$ condition.

This translates our $b \ge 2$, $n$, $j=2q$ problem, summary \ref{b2}, to
our $b' \ge 1$, $n-2$, $j-2 = 2q-2$ 
problem, summary \ref{b1}.  They are identical,
so, by induction, having already solved the later problem,
we solve the present problem.

Because of the idiosyncrasies of the $b=0$ case, we couldn't
just go from $b\ge 1$ to $b' \ge 0$, but had to do the
induction from $b \ge 2$ to $b' \ge 1$.

Because we must use $b=2$ and we have $2b+2 \le j+1$
and we must have $j=2q$, 
our lowest computation here is for $E_{1,7}$, so, to
use induction, we needed to have computed our $E_{1,5}$,
which we did in the previous section.
\end{proof}

\begin{remark}
	Rather than the downward induction we have done, we could equally
	well have done an induction on $b$.
	All that would be necessary would be to replace the $2$ in \ref{b2}
	with a $k$ and do the induction on $k$.  The statement of the
	excluded terms would be a bit more complicated and showing that the
	lower $i_t$ aren't involved would also be a bit more complicated.
	But, on the whole, the argument would be roughly equivalent.
\end{remark}

\section{All the $MU(n)$ theorems}

\begin{proof}[Proofs of theorems \ref{mu} and \ref{muodd}]
We begin with $n=2q$.  
In theorem \ref{bigthm}, for the part with $\vhat_1 = 0$,
we have $i_n = i_{n+1}$, but since there is no $i_{n+1}$,
this cannot happen and there is no contribution to the
answer from this second part.  We apply equation \ref{tochern}
to the $\Zq[\vhat_1]$ free part of theorem \ref{bigthm}.
Since $s(\ep) = 0$, we have $\voe = 1$.  We get,
modulo higher filtrations,
$$
\ww = 
\chat_1^{2i_1  -2i_2 }
 \chat_2^{2i_2  -2i_3 }
 \cdots
 \chat_n^{2i_n  }
=
\Pntr_1^{i_1  -i_2 }
 \Pntr_2^{i_2  -i_3 }
 \cdots
 \Pntr_n^{i_n  }
$$
We have
$i_{2b-1} = i_{2b}$  
for all $0 < b \le q$, so we end up with
$$
\pntr_2^{i_2}
\pntr_4^{i_4}
\cdots
\pntr_{2q}^{i_{2q}}.
$$
Of course, \A\ requires that $i_{2q} > 0$.
This gives us the $E_2$ of theorem \ref{mu}.

Moving on to $d_3$, because there is no $\uhat^\ep$
anymore and all the $\pntr_k$ are permanent cycles,
all of our $\ww$ for $E_2$ are permanent cycles.
Our entire $d_3$ is given by what happens on the
coefficient ring.
Using remark \ref{coef}, $d_3(v_2^2) = \vhat_1 v_2^{-4}$,
we get the $E_4$ term and the $x^3$-torsion generators.
The differential $d_7$ is again all on the coefficients
so we have $d_7(v_2^4) = \vhat_2 v_2^{-8} = 1$, and
we our $x^7$-torsion generators.

The proof for the $n=2q+1$ case is a bit different.
We can eliminate the $\Zq[\vhat_1]$ free part from
consideration because it requires $i_n = i_{n+1}$ and
there is no $i_{n+1}$.
We also have $\voe = 1$.
The reduction to Pontryagin classes is the same idea,
but our differential on the coefficients
$d_3(v_2^2) = \vhat_1 v_2^{-4}$
gives us a $\vhat_1$ that we don't have.  In our
$\ww$ we want to apply our usual relation \ref{oldformula2},
but if we do that, we must be sure that the resulting
$w_{I+\Delta_k,0}$ exists.  If $i_{2b} > i_{2b+1}$ we
can just use $\vhat_1 \pntr_{2b+1}^{i_{2b+1}} =
 \pntr_{2b+1}^{i_{2b+1}+1}$.  Anything lower than that
does not exist.
If, however,  
 $i_{2b} = i_{2b+1}$, we
cannot do that but we can use
$\vhat_1 \pntr_{2b'+1}^{i_{2b'+1}} =
 \pntr_{2b'+1}^{i_{2b'+1}+1}$  
where we have $b'$ is the smallest number with
$i_{2b'+1} = \cdots = i_{2b+1}$.  
This has \A\ and takes an element of type $b$
to one of type $b'$.
This allows us to hit all elements except when $b = q$ and
$i_{2q+1} = 1$.
This gives us both our $x^3$-torsion description and our
$E_4$
term of theorem \ref{muodd}.
There is no mystery now to 
$d_7$ or the
$x^7$-torsion. 
This is just computed on the coefficients as with $n = 2q$.
\end{proof}

\begin{remark}
	All the terms in the theorems that are in degrees
	$8*$ can be found just by eliminating the $v_2^{2,6}$.
	To see degrees $16*$, eliminate the $v_2^4$ as well.
	All $x^3$-torsion generators are in degrees $8*$
	and the $x^7$-torsion generators are in degrees $16*$.
	Since $x^i$-generators inject to $E(2)^*(-)$, this
	concludes the proof of theorem \ref{inj}.
\end{remark}

All that remains is to give a more $MU(n)$ associated
description of the $x^1$-torsion generators.  They are
all recoverable from theorem \ref{bigthm} where they
are written in terms of symmetric functions.  Here,
we rewrite this in terms of Pontryagin and Chern
classes to give it more the look of $MU(n)$.
Again, we rely on equation \ref{tochern}.
We can just read this off from \ref{bigthm}.

Recall from lemma \ref{parity} that when we write
our elements in terms of Chern classes, our
$\voe$ is determined by the parity of $j_1 + j_3 + j_5 + \cdots$,
for $\chat^J$.

\begin{thm}
\label{mux1}
Representatives for the $x^1$-torsion generators in
the associated graded object for $ER(2)^*(MU(n))$
start with:
$$
\Zp[\vhat_1]
[[\chat_1,\chat_2,\ldots,\chat_{n}]]
\{2 \voe v_2^{0,2,4,6} \chat_n \}
\iso
\Zp[\vhat_1]
[[\chat_1,\chat_2,\ldots,\chat_{n}]]
\{\voe \alpha_i \chat_n \}
\quad 0 \le i < 4
$$
For $ 1 \le j = 2b+1 \le n$, we have
$$
\Zq[\vhat_1][[\Pntr_2^{i_2}, \Pntr_4^{i_4},
\cdots, \Pntr_{2b}^{i_{2b}},
\Pntr_j^{i_j}, \chat_{j+1}^{i_{j+1}},
\chat_{j+2}^{i_{j+2}},
\cdots,
\chat_n^{i_n}]]\{\voe v_2^{0,2,4,6} \vhat_1 \Pntr_j \chat_n\}
$$ 
except when $j = n$, then we do not need the $\chat_n$ at
the end.
The parity that determines $\voe$ is the parity of
$j_{2b+3} + j_{2b+5} + j_{2b+7} + \cdots$.

For $0 \le 2b  < j = 2q \le n$ we get
{\small
$$
\Zq[\Pntr_2^{i_2}, \Pntr_4^{i_4},
\cdots, \Pntr_{2b}^{i_{2b}},
\Pntr_{2b+1}^{i_{2b+1}},
\Pntr_{2b+3}^{i_{2b+3}},
\cdots,
\Pntr_{j-1}^{i_{j-1}}, 
\Pntr_{j}^{i_{j}}, 
\chat_{j+1}^{i_{j+1}},
\chat_{j+2}^{i_{j+2}},
\cdots,
\chat_n^{i_n}]\{\voe v_2^{0,2,4,6} \Pntr_{2b+1} \Pntr_{j} \chat_n\}
$$}
\noindent
\hspace{-7pt}
except when $j = n$, then we do not need the $\chat_n$ at
the end.
The parity that determines $\voe$ is the parity of
$j_{2q+3} + j_{2q+3} + j_{2q+5} + \cdots$.
\end{thm}

\begin{remark}
To get the $x^1$-torsion generators in degrees $8*$, we have
to have $\voe = 1$ and we only use $v_2^{0,4}$.  For degrees
$16*$, we must have $\voe =1$ and no powers of $v_2$.
\end{remark}


\end{document}